\documentclass[a4paper]{article}
\usepackage{graphicx} 

\usepackage[numbers]{natbib}
\usepackage{natbibspacing}
\renewcommand{\bibfont}{\small}

\usepackage[mlmodern, varphi, scr-cats, sf-cats]{basemath}
\usepackage[vmargin=1in, hmargin=1.6in]{geometry}

\usepackage{bbm}
\usepackage{xcolor}

\usepackage{caption}
\usepackage{subcaption}

\newcommand{\demph}{\textbf}

\theoremstyle{plain}
\newtheorem{theorem}{Theorem}[section]
\newtheorem*{theorem*}{Theorem}
\newtheorem{lemma}[theorem]{Lemma}
\newtheorem{proposition}[theorem]{Proposition}
\newtheorem{corollary}[theorem]{Corollary}

\theoremstyle{definition}
\newtheorem{definition}[theorem]{Definition}
\newtheorem{remark}[theorem]{Remark}
\newtheorem{example}[theorem]{Example}
\newtheorem{examples}[theorem]{Examples}

\tikzset{shorten <>/.style={shorten >=#1,shorten <=#1}}
\tikzset{close/.style={outer sep=-1ex}}

\newcommand{\Prod}{\mathsf{Prod}}
\newcommand{\Pro}{\mathsf{Pro}}

\newcommand{\fdVect}{\mathsf{fdVect}}
\newcommand{\Stone}{\mathsf{Stone}}

\newcommand{\MND}{\mathsf{MND}}
\newcommand{\EM}{\mathsf{EM}}
\newcommand{\Lax}{\mathsf{Lax}}
\newcommand{\Colax}{\mathsf{Colax}}

\newcommand{\El}{\mathrm{El}}

\newcommand{\PS}{\mathscr{P}}
\newcommand{\calF}{\mathscr{F}}
\newcommand{\pf}[1]{{\uparrow}#1}

\renewcommand{\frak}[1]{\mathfrak{#1}}
\newcommand{\p}{\mathfrak{p}}

\newcommand{\rmf}{\mathsf{f}}
\newcommand{\rms}{\mathsf{s}}

\DeclareMathOperator{\Frac}{Frac}

\newcommand{\CHaus}{\mathsf{CHaus}}

\DeclareFontFamily{U}{mathx}{}
\DeclareFontShape{U}{mathx}{m}{n}{<-> mathx10}{}
\DeclareSymbolFont{mathx}{U}{mathx}{m}{n}
\DeclareMathAccent{\widehat}{0}{mathx}{"70}
\DeclareMathAccent{\widecheck}{0}{mathx}{"71}

\title{Pushforward monads}
\author{Adrián Doña Mateo\thanks{University of Edinburgh, \href{mailto:adrian.dona@ed.ac.uk}{adrian.dona@ed.ac.uk}}}
\date{}

\begin{document}

\maketitle

\begin{abstract}
Given a monad $T$ on $\cat{A}$ and a functor $G \colon \cat{A} \to \cat{B}$, one can construct a monad $G_\#T$ on $\cat{B}$ subject to the existence of a certain Kan extension; this is the pushforward of $T$ along $G$. We develop the general theory of this construction in a $2$-category, giving two universal properties it satisfies. In the case of monads in $\CAT$, this gives, among other things, two adjunctions between categories of monads on $\cat{A}$ and $\cat{B}$. We conclude by computing the pushforward of several familiar monads on the category of finite sets along the inclusion $\FinSet \hookrightarrow \Set$, which produces the monad for continuous lattices, among others. We also show that, with two trivial exceptions, these pushforwards never have rank.
\end{abstract}

\section*{Introduction}

The question of how a monad $T$ on $\cat{A}$ can be transported along a functor $G \colon \cat{A} \to \cat{B}$ to produce a monad on $\cat{B}$ has a well-known answer when $G$ has a left adjoint $F$, in which case $GTF$ has a monad structure induced from that of $T$ and the adjunction. There is a lesser-known, more general answer however: even when $G$ is not a right adjoint, the right Kan extension of $GT$ along $T$, when it exists, has a monad structure. The monad obtained thus is the pushforward of $T$ along $G$, which we denote by $G_\# T$. This notion of pushforward monad can be traced back to Street's seminal paper~\cite[p.~155]{Street1972}, but has received little attention since. In this paper, we explore the properties of this construction, and study several examples.

Pushforward monads generalise the concept of codensity monads, which correspond to the case when the monad on the domain category is the identity. This case has been extensively studied, with the most famous example being the codensity monad of the inclusion $\FinSet \hookrightarrow \Set$. It is a result of Kennison and Gildenhuys~\cite[p.~341]{Kennison1971} that this is the ultrafilter monad, which Manes~\cite[Prop.~5.5]{Manes1969} identified as the monad for compact Hausdorff spaces. This is a remarkable fact: the theory of compact Hausdorff spaces is derived from the concept of finite set and some general categorical machinery, without mention of topology. An account of this and related results was given by Leinster in~\cite{Leinster2013}. Other known codensity monads of functors without a left adjoint include:
\begin{itemize}
    \item The codensity monad of the inclusion of finite groups into groups is the profinite completion monad, whose algebras are profinite groups, or equivalently, totally disconnected compact Hausdorff topological groups.
    \item The codensity monad of the inclusion of finite rings into rings is again the profinite completion monad, whose algebras are profinite rings. These are all of the compact Hausdorff topological rings, which are automatically totally disconnected. This and the previous result can be found in Devlin's PhD thesis~\cite[Thms.~10.1.30 and~10.2.2]{Devlin2016}.
    \item The Giry monad, which sends a measurable space to the space of all probability measures on it, is the codensity monad of the inclusion of a certain (not full) subcategory of the category of measurable spaces. This and a number of similar results applying to other probability monads can be found in Avery's~\cite[Thm.~5.8]{Avery2016} and Van Belle's~\cite[]{VanBelle2022}.
\end{itemize}

All of these are examples of straightforward functors whose codensity monad gives an interesting, nontrivial theory. Pushforward monads add a new axis to this list, so that one can now vary both the functor and the monad on the domain being pushed forward.

In Section~\ref{sec:general}, we develop the theory of pushforward monads in a general $2$-category, starting from their construction. It includes a universal property satisfied by pushforwards with respect to lax transformations of monads, which was stated by Street in~\cite[Thm.~5]{Street1972}. This allows us to not only see codensity monads as special cases of pushforwards, but also vice versa. We also identify a new kind of universal property of pushforwards with respect to colax transformations of monads, generalising some ideas already present in Ad\'amek and Sousa's work on $D$-ultrafilters~\cite{Adamek2021}.

In Section~\ref{sec:CAT}, we specialise these results to the $2$-category $\CAT$ of locally small categories. We give sufficient conditions for the existence of pushforwards and for the possibility of their iteration. One remarkable consequence of the results in Section~\ref{sec:general} is the existence of two adjunctions between categories of monads, the most important one being the content of Theorem~\ref{thm:mnd_adj}. The last subsection shows that codensity monads are stable (in a suitable sense) under limit completions, and relates them to Diers's theory of multiadjunctions~\cite{Diers1980, Diers1981}, and Tholen's $\frak{D}$-pro-adjunctions~\cite{Tholen1984}. This allows us, for example, to identify the category of algebras of the codensity monad of the inclusion $\Field \hookrightarrow \Ring$ as the free product completion of the category of fields, $\Prod(\Field)$.

Lastly, in Section~\ref{sec:finset}, we study explicit examples of pushforward monads. Building on Kennison and Gildenhuys's result, we examine the pushforward of three different families of monads on $\FinSet$ along the inclusion $\FinSet \hookrightarrow \Set$. The most striking of these is the pushforward of the powerset monad, which we identify as the filter monad, whose algebras are continuous lattices. As with the ultrafilter monad, this example highlights the complexity of the process: the inputs are the concept of finite set and the (finite) powerset monad, and the output is the nontrivial theory of continuous lattices. The results in this section may be summarised as follows.

\begin{theorem*}
    Let $i \colon \FinSet \hookrightarrow \Set$ be the inclusion. Then:
    \begin{enumerate}[(i)]
        \item if $E$ is a finite set, then $\Set^{i_\#((\cdot) + E)}$ is the category of $E$-pointed compact Hausdorff spaces;
        \item if $M$ is a finite monoid, then $\Set^{i_\#(M \times (\cdot))}$ is the category of compact Hausdorff spaces with a discrete $M$-action;
        \item $\Set^{i_\# \PS}$ is the category of continuous lattices.
    \end{enumerate}
\end{theorem*}

\vspace{1em}
\noindent \textbf{Acknowledgements.} I would like to thank Tom Leinster for his invaluable guidance and many constructive discussions, and for introducing me to the concept of pushforward monads, which he in turn learned about from Gejza Jen\v{c}a. I would also like to thank Nathanael Arkor for his comments and questions, and for pointing me to the work of Diers and Tholen. Lastly, I would like to thank Zev Shirazi, who first told me about monad lifts and with whom I had many fruitful conversations during CT2024. This work has been funded by a PhD Scholarship from The Carnegie Trust for the Universities of Scotland.

\section{Pushforward monads}
\label{sec:general}

In this section, we spell out the construction of the pushforward of a monad and give its general properties. We work in the setting of a $2$-category $\cat{K}$, although from Section \ref{sec:CAT} onwards we will always take $\cat{K}$ to be $\CAT$. The pushforward construction makes sense in any bicategory, but we will do everything in the strict case for simplicity. We write $\alpha \cdot \beta$ for the vertical composite of two $2$-cells $\alpha$ and $\beta$, and $\alpha * \beta$ for their horizontal composite.

Given a $0$-cell $z$ and a $1$-cell $g \colon x \to y$ in $\cat{K}$, we write $g_* \colon \cat{K}(z,x) \to \cat{K}(z,y)$ and $g^* \colon \cat{K}(y,z) \to \cat{K}(x,z)$ for the functors given by composing with $g$ on the left and on the right, respectively. Given a monad $t$ on $x$, and taking $z = y$, we can form the comma category $g^*{\downarrow}gt$, whose objects are pairs $(s,\sigma)$ such that $s \in \cat{K}(y,y)$ and $\sigma \colon sg \to gt$.

\begin{proposition}\label{prop:comma_monoidal}
    Let $t$ be a monad on $x$, and $g \colon x \to y$ be a $1$-cell. The comma category $g^*{\downarrow}gt$ admits a strict monoidal structure such that the forgetful functor $g^*{\downarrow}gt \to \cat{K}(y,y)$ is strict monoidal.
\end{proposition}
\begin{proof}
    Let us write $\eta$ and $\mu$ for the unit and multiplication $2$-cells of $t$, respectively. Given two objects $(s,\sigma)$ and $(s',\sigma')$ of $g^*{\downarrow}gt$, we define $(s',\sigma') \otimes (s,\sigma)$ to be the composite:
    \[\begin{tikzcd}[column sep=3em, row sep=2.7em]
    x \ar[r, "g"] \ar[d, "t"] \ar[dd, bend right=50, ""{name=N, right}, "t"'] &
    y \ar[d, "s"] \ar[dl, shorten <>=16.5pt, Rightarrow, pos=0.4, "\sigma"']\\
    x \ar[r, "g"] \ar[d, "t"] \ar[Rightarrow, shorten <>=1pt, to=N, pos=0.4, "\mu"'{yshift=0.5ex}] &
    y \ar[d, "s'"] \ar[dl, shorten <>=16.5pt, Rightarrow, pos=0.4, "\sigma'"'{xshift=0.5ex}]\\
    x \ar[r, "g"] & y
    \end{tikzcd}\]
    The monoidal unit is simply $g \eta \colon g \to gt$. A routine calculation shows that the axioms for a strict monoidal category follow from the axioms for the monad $t$. That the forgetful functor $g^*{\downarrow}gt \to \cat{K}(y,y)$ is strict monoidal is clear.
\end{proof}

A right extension of $gt$ along $g$ is precisely a terminal object of $g^*{\downarrow}gt$. As the terminal object of a monoidal category, it has a unique monoid structure. Taking its image under the forgetful functor $g^*{\downarrow}gt \to \cat{K}(y,y)$ gives a monoid in $\cat{K}(y,y)$, i.e. a monad on $y$.

\begin{definition}\label{def:pushforward}
    Let $t$ be a monad on $x$ and $g \colon x \to y$ be a $1$-cell. The \demph{pushforward} of $t$ along $g$ is the right extension of $gt$ along $g$, with its canonical monad structure described above. When it exists, we denote it by $g_\# t$.
\end{definition}

We can give a more explicit description of the monad structure of $g_\#t$. Let $(g_\#t,\eps)$ be a right extension of $gt$ along $g$, as depicted by the following diagram.
\[\begin{tikzcd}[column sep=3em, row sep=2.7em]
    x \ar[r, "g"] \ar[d, "t"'] & y \ar[d, "g_\#t"] \ar[dl, shorten <>=16.5pt, Rightarrow, "\eps"', pos=0.4] \\
    x \ar[r, "g"'] & y
\end{tikzcd}\]
We will usually call $\eps$ the \demph{counit} of the pushforward $g_\#t$. The unit and multiplication of $g_\# t$ are the unique $2$-cells $\eta^{g_\#t}$ and $\mu^{g_\#t}$ making
\begin{equation}\label{eq:monad_struct}
  \begin{tikzcd}[column sep=0.8em, row sep=2.7em]
    & g \ar[dl, "(\eta^{g_\#t}) g"'] \ar[dr, "g \eta^t"] \\
    (g_\#t)g \ar[rr, "\eps"'] && gt
  \end{tikzcd}
  \qquad \text{and} \qquad
  \begin{tikzcd}[column sep=3em, row sep=2.7em]
    (g_\#t)^2 g \ar[d, "\mu^{g_\#t}g"'] \ar[r, "(g_\#t)\eps"] & (g_\#t)gt \ar[r, "\eps t"] & gt^2 \ar[d, "g\mu^t"] \\
    (g_\#t)g \ar[rr, "\eps"'] && gt
  \end{tikzcd}
\end{equation}
commute.

\begin{examples}\label{ex:pushforward}
    \begin{enumerate}[(i)]
    \item[]
    \item \label{ex:pf:adjoint}
    If $g \colon x \to y$ is right adjoint to $f$, with unit $\delta$ and counit $\nu$, then the right extension of a $1$-cell $h \colon x \to z$ along $g$ is $(hf, h\nu)$. In particular, $g_\# t = gtf$, with its well-known monad structure:
    \[\eta^{g_\# t} = g\eta^t f \cdot \delta \qquad \text{and} \qquad \mu^{g_\# t} = g \mu^t f \cdot gt \nu tf.\]

    \item \label{ex:pf:codensity}
    The pushforward of the identity monad along $g$ is the \demph{codensity monad} of $g$. In this case, the underlying endofunctor of the monad is the right extension of $g$ along itself. By the previous example, the codensity monad of a right adjoint is the monad induced by the adjunction. The most famous example of a codensity monad in the $2$-category $\CAT$ of a functor that is not a right adjoint is that of the inclusion functor $\FinSet \hookrightarrow \Set$, which was identified as the ultrafilter monad by Kennison and Gildenhuys~\cite[p.~341]{Kennison1971}.
    
    \item \label{ex:pf:terminal}
    Let $\mathsf{1}$ be the terminal category. There is a unique monad on $\mathsf{1}$, namely the identity monad. A functor $g \colon \mathsf{1} \to \cat{C}$ corresponds to an object $x$ of $\cat{C}$. If $\cat{C}$ is locally small and has $\Set$-indexed powers, then the codensity monad of $g$ is given by $y \mapsto x^{\cat{C}(y,x)}$. This is called the \demph{endomorphism monad} of $x$. Proposition~\ref{thm:univ} and the discussion following it show that for any monad $T$ on $\cat{C}$, the $T$-algebra structures on $x$ correspond to monad maps (see~\ref{def:monad_morphisms}\ref{def:mm:map}) from $T$ to the endomorphism monad of $x$.

    \item Let $\cat{K}$ be a monoidal category, thought of as a one-object bicategory. Monads in $\cat{K}$ are simply monoids, and right extension along an object $x \in \cat{K}$ is by definition a right adjoint to $- \otimes x$. If $m$ is a monoid, and $- \otimes x \dashv [x,-]$, then the pushforward $x_\#m$ is $[x,x\otimes m]$, with unit and multiplication the respective transposes of
    \[
    \begin{tikzcd}
      x \ar[d, "x \otimes \eta^m"] \\
      x \otimes m
    \end{tikzcd}
    \qquad \text{and} \qquad
    \begin{tikzcd}
      {[x,x\otimes m]} \otimes [x,x\otimes m] \otimes x \ar[d, "{[x,x\otimes m] \otimes \mathrm{ev}_{x \otimes m}}"] \\
      {[x,x\otimes m]} \otimes x \otimes m \ar[d, "{\mathrm{ev}_{x \otimes m} \otimes m}"] \\
      x \otimes m \otimes m \ar[d, "x \otimes \mu^m"] \\
      x \otimes m.
    \end{tikzcd}\]
    \end{enumerate}
\end{examples}

The pushforward construction was first introduced by Street in~\cite[p.~155]{Street1972}. In fact, he defines it in terms of a universal property (given in Theorem~\ref{thm:univ}) with respect to the forgetful functor from the $2$-category of monads in $\cat{K}$, denoted by $\MND(\cat{K})$, to $\cat{K}$, and then shows that the right extension in Definition~\ref{def:pushforward} has this universal property. In~\cite{Lack2002}, Lack and Street defined a related $2$-category $\EM(\cat{K})$ which is to be thought of as the free completion of $\cat{K}$ under Eilenberg--Moore objects. These objects generalise the usual Eilenberg--Moore categories of a monad in $\CAT$ to the context of $2$-categories, and are given by a particular weighted limit (identified by Street in~\cite[p.~178]{Street1976}). The $2$-categories $\MND(\cat{K})$ and $\EM(\cat{K})$ have the same $0$- and $1$-cells, but the latter has more $2$-cells in general. We will not be concerned with the $2$-cells, so for simplicity we will only refer to $\EM(\cat{K})$. Next, we introduce its $1$-cells with some standard terminology.

\begin{definition}\label{def:monad_morphisms}
    Let $t$ be a monad on $x$, and $s$ be a monad on $y$.
    \begin{enumerate}[(i)]
        \item \label{def:mm:lax} A \demph{lax transformation} of monads from $t$ to $s$ is a pair $(g, \phi)$ of a $1$-cell $g \colon x \to y$ and a $2$-cell $\phi \colon sg \to gt$ such that 
        \begin{equation}\label{eq:lax}
            \phi \cdot \eta^s g = g\eta^t \qquad \text{and} \qquad \phi \cdot \mu^s g = g \mu^t \cdot \phi t \cdot s \phi.
        \end{equation}
        Equivalently, it is a $1$-cell $(x,t) \to (y,s)$ in $\EM(\cat{K})$. Let $\Lax(t,s)$ denote the (large) set of lax transformations from $t$ to $s$, and $\Lax_g(t,s)$ denote the subset of those whose $1$-cell part is $g$.

        \item \label{def:mm:colax} A \demph{colax transformation} of monads from $t$ to $s$ is a pair $(g,\psi)$ of a $1$-cell $g \colon x \to y$ and a $2$-cell $\psi \colon gt \to sg$ such that
        \begin{equation}\label{eq:colax}
            \psi \cdot g \eta^t = \eta^s g \qquad \text{and} \qquad \psi \cdot g \mu^t = \mu^s g \cdot s \psi \cdot \psi t.
        \end{equation}
        Equivalently, it is a $1$-cell $(x,t) \to (y,s)$ in $\EM(\cat{K}^\opp)^\opp$. Let $\Colax(t,s)$ denote the (large) set of colax transformations from $t$ to $s$, and $\Colax_g(t,s)$ denote the subset of those whose $1$-cell part is $g$.

        \item \label{def:mm:map} If $x = y$, a \demph{map of monads} on $x$ (or simply a \demph{monad map}) from $t$ to $s$ is a $2$-cell $\theta \colon t \to s$ such that
        \begin{equation}\label{eq:mm}
            \theta \cdot \eta^t = \eta^s \qquad \text{and} \qquad \theta \cdot \mu^t = \mu^s \cdot (\theta * \theta).
        \end{equation}
        Let $\Mnd_x(t,s)$ denote the (large) set of monad maps from $t$ to $s$. Note that $\Mnd_x(t,s) = \Lax_{1_x}(s,t) = \Colax_{1_x}(t,s)$.
    \end{enumerate}
\end{definition}

\begin{examples}\label{ex:monad_morphisms}
    Let $t$ be a monad on $x$.
    \begin{enumerate}[(i)]
        \item \label{ex:mm:eta} The unit $\eta^t$ is a monad map $1_x \to t$.
        \item \label{ex:mm:mu} The pair $(t,\mu^t)$ is a lax transformation $1_x \to t$, since in this case the aximos in~\eqref{eq:lax} become the left unitality and associativity of $t$. Dually, $(t,\mu^t)$ is a colax transformation $t \to 1_x$.
        \item \label{ex:mm:eps} Let $g \colon x \to y$ be a $1$-cell such that $g_\#t$ exists and has counit $\eps$. Then~\eqref{eq:monad_struct} shows that $(g,\eps)$ is a lax transformation $t \to g_\# t$. In fact, there is a natural bijection between monoids in $g^*{\downarrow}gt$ (with its strict monoidal structure given by Proposition~\ref{prop:comma_monoidal}) and monads on $y$ equipped with a lax transformation from $t$ whose $1$-cell part is $g$.
    \end{enumerate}
\end{examples}

With these definitions, $\Lax$ and $\Colax$ become $1$-categories whose objects are monads in $\cat{K}$, namely the $1$-truncations of $\EM(\cat{K})$ and $\EM(\cat{K}^\opp)^\opp$, respectively. In general, these categories are not even locally small. The composite of $(g,\phi) \in \Lax(t,s)$ and $(g',\phi') \in \Lax(s,r)$ is the pair $(g'g, \phi'g \cdot g'\phi)$. Dually, the composite of $(g,\psi) \in \Colax(t,s)$ and $(g',\psi) \in \Colax(s,r)$ is the pair $(g'g, g'\psi \cdot \psi'g)$. In particular, $\Mnd_x$ becomes a $1$-category, whose objects are monads on $x$, and which is a subcategory of $\cat{K}(x,x)$ and $\Colax$, and contravariantly of $\Lax$.

The pushforward construction is a functor in the monad argument, and in fact one which is as continuous as the $1$-cell along which on one pushes forward, as given by the next lemma. It is not quite functorial in the $1$-cell argument, but it does preserve composition laxly (see Lemma~\ref{lem:pf_preserve}).

\begin{lemma}\label{lem:pf_functor}
    Let $g \colon x \to y$ be a $1$-cell. Then $g_\#$ is a functor $\Mnd_x \to \Mnd_y$ insofar as it is defined. Moreover, $g_\#$ preserves the limit of any diagram $D \colon \cat{I} \to \Mnd_x$ such that $g_*$ preserves the limit of $\cat{I} \xrightarrow{D} \Mnd_x \to \cat{K}(x,x)$.
\end{lemma}
\begin{proof}
    Let $t$ and $s$ be monads on $x$ such that $g_\#t$ and $g_\#s$ exist, with counits $\eps^t$ and $\eps^s$ respectively. Given $\theta \in \Mnd_x(t,s)$, there is a unique $2$-cell $g_\#\theta \colon g_\#t \to g_\#s$ such that $\eps^s \cdot (g_\# \theta) g = g\theta \cdot \eps^t$. This is a monad map: the equation
    \[\eps^s \cdot (g_\# \theta \cdot \eta^{g_\#t}) g = g\theta \cdot \eps^t \cdot \eta^{g_\#t} g \stackrel{\eqref{eq:monad_struct}}{=} g\theta \cdot g\eta^t \stackrel{\eqref{eq:mm}}{=} g\eta^s \stackrel{\eqref{eq:monad_struct}}{=} \eps^s \cdot \eta^{g_\#s} g\]
    implies that $g_\# \theta \cdot \eta^{g_\#t} = \eta^{g_\#s}$ by uniqueness of factorisations through $\eps^s$. The proof that $g_\#\theta$ preserves multiplication is similar. As is often the case, the uniqueness of $g_\#\theta$ guarantees functoriality.

    For the statement about of limits, recall that the functor $U_x \colon \Mnd_x \to \cat{K}(x,x)$ creates (and hence preserves and reflects) limits, being the forgetful functor from a category of monoids in a monoidal category. Let $(l, \lambda)$ be a limit cone for $D$. Then $(g_* U_x l, g_* U_x \lambda)$ is a limit cone for $g_* U_x D$ in $\cat{K}(x,y)$ by assumption. Since $\ran_g$ is a partial right adjoint to $g^*$, it preserves limits whenever it is defined, so $(\ran_g g_* U_x l, \ran_g g_* U_x \lambda)$ is still a limit cone. But this is precisely the image under $U_y$ of the cone $(g_\# l, g_\# \lambda)$, which is then a limit cone since $U_y$ reflects limits.
\end{proof}

The following theorem was stated without proof by Street, albeit in a slightly different language. It gives the universal property of pushforward monads.

\begin{theorem}[{Street~\cite[Thm.~5]{Street1972}}]\label{thm:univ}
    Let $g \colon x \to y$ be a $1$-cell, $t$ be a monad on $x$, and $s$ be a monad on $y$. If $g_\# t$ exists, then there is a bijection
    \[\Lax_g(t,s) \cong \Mnd_y(s, g_\# t)\]
    natural in $t$ and $s$.
\end{theorem}
\begin{proof}
    Let $\eps$ be the counit of $g_\# t$, so that $(g,\eps) \in \Lax(t,g_\#t)$ by Example~\ref{ex:monad_morphisms}\ref{ex:mm:eps}. A monad map $\theta \colon s \to g_\#t$ is in particular a lax transformation $(1_y,\theta) \colon g_\#t \to s$, so composing with $(g,\eps)$ gives an assignment $\Mnd_y (s,g_\#t) \to \Lax_g(t,s)$. Given $(g,\phi) \in \Lax_g(t,s)$, the universal property of the right extension defining $g_\#t$ gives a unique $2$-cell $\widehat\phi \colon s \to g_\#t$ such that $\phi = \eps \cdot \widehat\phi g$. It suffices to check that $\widehat\phi$ is a monad map, since then this construction gives an inverse to the former assignment. We have
    \[\eps \cdot (\widehat\phi \cdot \eta^s)g = \eps \cdot \widehat\phi g \cdot \eta^s g = \phi \cdot \eta^s g \stackrel{\eqref{eq:lax}}{=} g\eta^t \stackrel{\eqref{eq:monad_struct}}{=} \eps \cdot \eta^{g_\#t}g,\]
    which implies that $\widehat\phi \cdot \eta^s = \eta^{g_\#t}$ by uniqueness of factorisations through $\eps$. The proof that $\widehat\phi$ preserves multiplication is similar.

    Naturality in $t$ is with respect to monad maps in $\Mnd_x$, while that in $s$ is with respect to monad maps in $\Mnd_y$. They both follow easily from the fact that the bijection is given by composition with $(g,\eps)$ and that $g_\#$ is a functor.
\end{proof}

The significance of this universal property becomes clear when $\cat{K}$ has Eilenberg--Moore objects. Recall that this is a completeness condition under a certain class of weighted limits, which is satisfied, for example, by $\CAT$.

\begin{proposition}[Lack, Street]\label{prop:lax_trans_EM_lift}
  Let $\cat{K}$ have Eilenberg--Moore objects, and write $x^t$ for the Eilenberg--Moore object of a monad $t$ on $x$. Let $t$ be a monad on $x$ and $s$ be a monad on $y$. Then there is a bijection between lax transformations of monads $(g,\phi) \in \Lax_g(t,s)$ and $1$-cells $g^\phi$ such that the square
  \[\begin{tikzcd}
      x^t \ar[r, "g^\phi"] \ar[d, "u^t"'] & y^s \ar[d, "u^s"] \\
      x \ar[r, "g"'] & y
  \end{tikzcd}\]
  commutes.
\end{proposition}
\begin{proof}
  This follows from the analysis of $\EM(\cat{K})$ by Lack and Street in~\cite[\S2.2]{Lack2002}.
\end{proof}

The bijection in Theorem~\ref{thm:univ} then becomes a correspondence between dashed $1$-cells as follows:
\begin{equation}\label{eq:monadic_reflection}
    \begin{tikzcd}
        x^t \ar[r, dashed, "\forall"] \ar[d, "u^t"'] & y^s \ar[d, "u^s"] \\
        x \ar[r, "g"'] & y
    \end{tikzcd}
    \qquad \qquad
    \begin{tikzcd}[column sep=5pt]
        y^{g_\# t} \ar[rr, dashed, "\exists !"] \ar[dr, "u^{g_\#t}"', near start] && y^s \ar[dl, "u^s", near start] \\
        & y
    \end{tikzcd}
\end{equation}
In other words, $1$-cells over $y$ from $gu^t$ to $u^s$ correspond to $1$-cells over $y$ from $u^{g_\#t}$ to $u^s$. The $1$-cells of the form $u^s$ are called \demph{monadic}, so this gives a way of `approximating' $gu^t$ by a monadic $1$-cell. 

If we consider the case where $t$ is the identity monad on $x$, then $u^t$ is an isomorphism, and $g_\# t$ is the codensity monad of $g$. Its universal property implies that the assignment $g \mapsto u^{g_\#1}$ gives a (partially defined) reflection from the slice category $\cat{K}/y$ to its full subcategory on the monadic $1$-cells. Thus, when $g_\#1$ exists, we speak of $u^{g_\#1}$ as the \demph{monadic reflection} of $g$.

\begin{lemma}\label{lem:pf_preserve}
    Let $g \colon x \to y$ and $h \colon y \to z$ be $1$-cells, and $t$ a monad on $x$. Then there exists a canonical map of monads on $z$
    \[h_\#(g_\# t) \to (hg)_\# t\]
    natural in $t$, assuming the domain and codomain exist. Moreover, this map is an isomorphism if $h$ preserves the right extension of $gt$ along $g$.
\end{lemma}
\begin{proof}
    Let $\eps^g$, $\eps^h$ and $\eps^{hg}$ be the counits of $g_\# t$, $h_\#(g_\#t)$ and $(hg)_\#t$, respectively. By Example~\ref{ex:monad_morphisms}\ref{ex:mm:eps}, we have $(g,\eps^g) \in \Lax(t, g_\#t)$ and $(h,\eps^h) \in \Lax(g_\#t, h_\#(g_\#t))$, so their composite $(hg, \eps^h g \cdot h \eps^g)$ is a lax transformation $t \to h_\#(g_\#t)$ whose $1$-cell part is $hg$. By Theorem~\ref{thm:univ}, this corresponds to a unique monad map $h_\#(g_\# t) \to (hg)_\#t$. Naturality follows from the fact that pushforward is a functor, and that the bijection in Theorem~\ref{thm:univ} is natural in $t$.

    If $h$ preserves $\ran_g gt$, then $\ran_g hgt = (h(\ran_g gt), h\eps^g)$. By general properties of extensions, $\ran_h (\ran_g hgt) = \ran_{hg} hgt$, but the former may now be given as $(\ran_h h(\ran_g gt), \eps^h g \cdot h \eps^g)$. Thus the factorisation of $\eps^hg \cdot h\eps^g$ through $\eps^{hg}$ is an isomorphism.
\end{proof}

Informally, one can interpret this lemma as saying that the assignment $x \mapsto \Mnd_x$ and $g \mapsto g_\#$ is a lax $2$-functor from $\cat{K}_0$ (the underlying $1$-category of $\cat{K}$) to a strict $2$-category of categories, (partial) functors and natural transformations.

\begin{remark}\label{rmk:pf_preserve}
    In the notation of Lemma~\ref{lem:pf_preserve}, for $h$ to preserve $\ran_g gt$ it suffices for one of $g$ and $h$ to be a right adjoint. This is because in any $2$-category, right extensions along right adjoints are absolute (i.e.\ preserved by any $1$-cell), and right adjoints preserve right extensions.

    In $\CAT$, it is also sufficient for the right Kan extension defining $g_\# t$ to be pointwise, and for $h$ to preserve the limits involved. This is true, for example, if $x$ is a small category, $y$ is complete and $h$ preserves all small limits.

    For an example where the map is not an isomorphism, one can take $\cat{K} = \CAT$, $g$ to be the unique functor $\mathsf{0} \to \mathsf{1}$ from the initial to the terminal category, $h$ to be the functor $\mathsf{1} \to \Set$ picking out a set $X$ with at least two elements, and $t$ to be the identity monad on $\mathsf{0}$. Then $g_\#t$ is the identity monad on $\mathsf{1}$, and $h_\#(g_\# t)$ is the endomorphism monad of $X$, as in Example~\ref{ex:pushforward}\ref{ex:pf:terminal}, while $(hg)_\# t$ is constant at the terminal set.
\end{remark}

Note that~\eqref{eq:monadic_reflection} says that $u^{g_\#t}$ is the monadic reflection of $gu^t$, but the same is true for $u^{(gu^t)_\#1}$ assuming $(gu^t)_\# 1$ exists. Since $u^t$ is a right adjoint, the previous remark and lemma imply that $(gu^t)_\#1 = g_\# (u^t_\#1)$, and by Example~\ref{ex:pushforward}\ref{ex:pf:adjoint} we have $u^t_\# 1 = t$. Altogether, this shows:

\begin{corollary}\label{cor:pf_codensity}
    Let $\cat{K}$ be a $2$-category with Eilenberg--Moore objects, $g \colon x \to y$ be a $1$-cell in $\cat{K}$, and $t$ a monad on $x$. Then $g_\#t$ is the codensity monad of $gu^t$.
\end{corollary}

This corollary seems to suggest that one need not think about pushforwards, and that studying codensity monads would suffice. However, the pushforward construction has many benefits that are lost by passing to the associated codensity monad. Most importantly, it gives a functor between categories of monads, and so it can be used to produce maps between complex monads from maps between simple ones.

\subsection{Colax transformations and \texorpdfstring{$g$}{g}-determined monads}\label{ssec:g-det}

In this subsection, we study a universal property that pushforwards have with respect to colax transformations of monads. This work was inspired by Ad\'amek and Sousa's paper~\cite{Adamek2021}, where they give explicit descriptions of the codensity monads of the inclusions $\cat{A}_\mathsf{fp} \hookrightarrow \cat{A}$ for several locally finitely presentable categories $\cat{A}$. They say a monad $T$ on $\cat{A}$ has the \textit{limit property} with respect to a full embedding $i \colon \cat{B} \hookrightarrow \cat{A}$ if $(T,1_{Ti})$ is the pointwise right Kan extension of $Ti$. They then exhibit a certain monad $T$ with this property, and show that the codensity monad of $i$ is the smallest submonad of $T$ with the limit property. Our results will generalise this to the context of pushforward monads in a $2$-category.

We will it find useful to dualise some of the results in this section. Formally, these duals are obtained by passing to the $2$-category $\cat{K}^\opp$ where $1$-cells have been reversed. Extensions then become lifts: a \demph{right lift} of a $1$-cell $g \colon x \to z$ through a $1$-cell $f \colon y \to z$ is a terminal object $(h, \eps)$ of $f_*{\downarrow}g$, which we will denote by $\rift_f g$. Dually, a \demph{left lift} of $g$ through $f$ is an initial object of $g{\downarrow}f_*$, denoted by $\lift_f g$. Lifts seem to be less common in the literature than extensions, but they are famously used in Street and Walters's theory of Yoneda structures~\cite[\S1]{Street1978}.

\begin{definition}\label{def:g-det}
    Let $g \colon x \to y$ and $h \colon y \to z$ be $1$-cells. We say that $h$ is \demph{$g$-determined} if $(h, 1_{hg})$ is a right extension of $hg$ along $g$. Dually, we say $g$ is \demph{$h$-opdetermined} if it is $h$-determined in $\cat{K}^\opp$, i.e.\ if $(g, 1_{hg}) = \rift_ h hg$. A monad is $g$-(op)determined if its $1$-cell part is.
\end{definition}

The condition in this definition can be replaced by an apparently weaker one, as shown by the next lemma.

\begin{lemma}\label{lem:g-det_iff}
    In the notation of the previous definition, $h$ is $g$-determined iff there exists a monic $2$-cell $\eps$ such that $\ran_g hg = (h, \eps)$.
\end{lemma}
\begin{proof}
    The forwards implication is clear. Now let $\ran_g hg = (h,\eps)$ with $\eps$ monic, and let $\alpha, \beta \colon h \to h$ be the unique $2$-cells such that $\eps \cdot \alpha g = \eps \cdot \eps$ and $\eps \cdot \beta g = 1_{hg}$. Since $\eps$ is monic, this implies that $\alpha g = \eps$. Then
    \[\eps \cdot (\alpha \cdot \beta)g = \eps \cdot \alpha g \cdot \beta g = \eps \cdot \eps \cdot \beta g = \eps\]
    and
    \[\eps \cdot (\beta \cdot \alpha)g = \eps \cdot \beta g \cdot \alpha g = 1_{hg} \cdot \alpha g = \eps,\]
    which imply that $\alpha$ and $\beta$ are mutual inverses, since $\eps$ is the counit of a right extension along $g$. It follows that $\beta$ is an isomorphism $(h, 1_{hg}) \cong (h, \eps)$ in $g^*{\downarrow}hg$, so $(h,1_{hg})$ is also a right extension.
\end{proof}

\begin{example}
    Let $g \colon x \to y$ be right adjoint to $f$, with unit and counit $\eta$ and $\eps$, respectively. Then $h \colon y \to z$ is $g$-determined iff $h\eta$ is an isomorphism. Indeed, we have $\ran_g hg = (hgf, hg\eps)$. The claim follows easily from the fact that $h\eta$ is the unique morphism $(h, 1_{hg}) \to (hgf, hg\eps)$ in $g^*{\downarrow}hg$, since the codomain is terminal.

    It follows that $f$ is $g$-determined iff the adjunction $f \dashv g$ is idempotent. In $\CAT$, we also have that if $G \colon \cat{A} \to \cat{B}$ is a coreflection, then any functor with domain $\cat{B}$ is $G$-determined.
\end{example}

\begin{theorem}\label{thm:colax_corr}
    Let $g \colon x \to y$ be a $1$-cell, $t$ be a monad on $x$, and $s$ be a $g$-determined monad on $y$. If $g_\# t$ exists, then right extending along $g$ gives a function
    \[\Colax_g(t,s) \to \Mnd_y(g_\# t, s)\]
    natural in $t$ and $s$. Moreover, it takes monic $2$-cells to monic $2$-cells.
\end{theorem}
\begin{proof}
    Let $\eps$ be the counit of $g_\#t$. Given $(g,\psi) \in \Colax_g(t,s)$, we have $\psi \colon gt \to sg$. Let $\widehat{\psi} \coloneqq \ran_g \psi \colon \ran_g gt \to \ran_g sg$. Since $s$ is $g$-determined, we may take the right extension in the codomain to be $(s,1_{sg})$, so that $\widehat{\psi}$ is the unique $2$-cell such that 
    \begin{equation}\label{eq:psi_hat}
        \widehat{\psi}g = \psi \cdot \eps.
    \end{equation}
    We check that $\widehat{\psi}$ is a monad map $g_\#t \to s$. We have
    \[(\widehat{\psi} \cdot \eta^{g_\#t})g \stackrel{\eqref{eq:psi_hat}}{=} \psi \cdot \eps \cdot \eta^{g_\#t}g \stackrel{\eqref{eq:monad_struct}}{=} \psi \cdot g\eta^t \stackrel{\eqref{eq:colax}}{=} \eta^s g,\]
    so $\widehat{\psi} \cdot \eta^{g_\#t} = \eta^s$, since they both have the same factorisation through the right extension $1_{sg}$. For the multiplication axiom we have
    \begin{align*}
        (\widehat{\psi} \cdot \mu^{g_\#t})g &= \psi \cdot \eps \cdot \mu^{g_\#t} g && \text{(by \eqref{eq:psi_hat})}\\
        &= \psi \cdot g \mu^t \cdot \eps t \cdot (g_\#t) \eps && \text{(by \eqref{eq:monad_struct})} \\
        &= \mu^s g \cdot s \psi \cdot \psi t \cdot \eps t \cdot (g_\#t) \eps && \text{(by \eqref{eq:colax})} \\
        &= \mu^s g \cdot s \psi \cdot \widehat{\psi}gt \cdot (g_\#t) \eps && \text{(by \eqref{eq:psi_hat})} \\
        &= \mu^s g \cdot \widehat{\psi} sg \cdot (g_\# t) \psi \cdot (g_\# t) \eps && \text{(by the interchange law)} \\
        &= \mu^s g \cdot \widehat{\psi} sg \cdot (g_\#t) \widehat{\psi} g && \text{(by \eqref{eq:psi_hat})} \\
        &= (\mu^s \cdot \widehat{\psi}s \cdot (g_\#t)\widehat{\psi}) g.
    \end{align*}
    For the same reason as before, this implies $\widehat{\psi} \cdot \mu^{g_\#t} = \mu^s \cdot \widehat{\psi}s \cdot (g_\#t)\widehat{\psi}$.

    As in Theorem~\ref{thm:univ}, naturality in $t$ is with respect to monad maps in $\Mnd_x$, and naturality in $s$ is with respect to monad maps in $\Mnd_y$ between $g$-determined monads. In both cases, it follows from the fact that $\ran_g$ is a (partial) functor, and that the monads on $y$ are $g$-determined.

    The preservation of monic $2$-cells is immediate from the fact that $\ran_g$ is a partial right adjoint.
\end{proof}

\begin{remark}\label{rmk:monic}
    If $\cat{K}(x,x)$ has kernel pairs (e.g.\ if $\cat{K} = \CAT$ and $x$ has kernel pairs), then a map of monads on $x$ is monic iff it is monic as a $2$-cell. This follows from the fact that the forgetful functor $\Mnd_x \to \cat{K}(x,x)$ creates limits, since a morphism is monic iff its kernel pair consists of identities.
\end{remark}

\begin{corollary}\label{cor:colax_univ}
    In the notation of Theorem~\ref{thm:colax_corr}, if the counit of $g_\#t$ is an isomorphism, then the function in that theorem is a bijection
    \[\Colax_g(t,s) \cong \Mnd_y(g_\#t, s).\]
\end{corollary}
\begin{proof}
    Let $\eps$ be the counit of $g_\#t$. It is easy to see from~\eqref{eq:lax} and~\eqref{eq:colax}, that for an invertible $2$-cell $\phi \colon sg \to gt$ we have $(g,\phi) \in \Lax(t,s)$ iff $(g,\phi^{-1}) \in \Colax(t,s)$. Then $(g,\eps^{-1}) \in \Colax_g(t,g_\#t)$, and an inverse to the function in Theorem~\ref{thm:colax_corr} is given by composing by $(g,\eps^{-1})$ on the right in the category $\Colax$.
\end{proof}

This corollary gives a second universal property of pushforwards whose counit is an isomorphism. In $\CAT$, this is the case as soon as $g$ is full and faithful and the right Kan extension is pointwise.

To finish off this section, we will take advantage of duality to produce analogous statements about monad lifts. Formally, we are replacing the $2$-category $\cat{K}$ with $\cat{K}^\opp$. This process only reverses the direction of the $1$-cells, so monads in $\cat{K}^\opp$ are the same thing as monads in $\cat{K}$. It does, however, swap the concepts of lax and colax transformations of monads.

\begin{definition}
    Let $g \colon x \to y$ be a $1$-cell, and $s$ be a monad on $y$. The \demph{monad lift} of $s$ along $g$ is the right lift of $sg$ along $g$, with its canonical monad structure. When it exists, we denote it by $g^\#s$.
\end{definition}

Monad lifts have been used by Shirazi~\cite[Thm.~4.5]{Shirazi2024} to give presentations of probability monads as codensity monads.

As with pushforwards, $g^\#$ is a functor $\Mnd_y \to \Mnd_x$ insofar as it is defined, which preserves those limits that $g^* \colon \cat{K}(y,y) \to \cat{K}(x,y)$ preserves. The dual of Theorem~\ref{thm:univ} shows that, given a $1$-cell $g \colon x \to y$, and monads $t$ on $x$ and $s$ on $y$, there is a natural bijection
\begin{equation}\label{eq:lift_univ}
    \Colax_g(t,s) \cong \Mnd_x(t,g^\# s)
\end{equation}
whenever $g^\# s$ exists. Moreover, the dual of Corollary~\ref{cor:colax_univ} shows that, if additionally $t$ is $g$-opdetermined and the counit of $g^\# s$ is an isomorphism, then there is a natural bijection
\begin{equation}\label{eq:lift_lax_univ}
    \Lax_g(t,s) \cong \Mnd_x(g^\# s, t).
\end{equation}

A moment's glance at these isomorphisms together with their duals immediately gives two adjunctions between two pairs of full subcategories of $\Mnd_x$ and $\Mnd_y$:
\begin{enumerate}[(i)]
    \item Theorem~\ref{thm:univ} and~\eqref{eq:lift_lax_univ} gives $g^\# \dashv g_\#$; and
    \item Corollary~\ref{cor:colax_univ} and~\eqref{eq:lift_univ} gives $g_\# \dashv g^\#$.
\end{enumerate}
The respective full subcategories of $\Mnd_x$ and $\Mnd_y$ are generated by the monads that satisfy the assumptions of the results in each case. These conditions do not appear very natural in a general $2$-category, however we will see in the next section that, at least in the case of the first adjunction, they are not rare when $\cat{K} = \CAT$.

\section{Pushforwards in \texorpdfstring{$\CAT$}{CAT}}
\label{sec:CAT}

In this section we study pushforward monads in the $2$-category $\CAT$ of locally small categories, functors and natural transformations. We give sufficient conditions for the existence of pushforwards, specialise and refine many of the results in the previous section, and show that codensity monads are invariant under limit completions.

\subsection{Existence}

In $\CAT$, extensions are usually called Kan extensions, and there are well-known formulas that compute them in terms of (co)limits in the codomain category. Let $F\colon \cat{A} \to \cat{X}$ and $G \colon \cat{A} \to \cat{B}$ be functors. If for each $b \in \cat{B}$ the (weighted) limit
\begin{equation}\label{eq:ptw_Ran}
\{\cat{B}(b,G-),F\} = \lim \left( b{\downarrow}G \xrightarrow{\Pi_b} \cat{A} \xrightarrow{F} \cat{X} \right)
\end{equation}
exists, where $\Pi_b$ is the forgetful functor, then the assignment $b \mapsto \lim F\Pi_b$ assembles into a functor $\cat{B} \to \cat{X}$, which is a (pointwise) right Kan extension of $F$ along $G$. The details of this can be found in Riehl's book~\cite[\S6]{Riehl2017}.

If $T$ is a monad on $\cat{A}$, then the pushforward $G_\#T$ is given by
\begin{equation}\label{eq:ptw_pf}
    G_\#T(b) = \lim \left( b{\downarrow}G \xrightarrow{\Pi_b} \cat{A} \xrightarrow{T} \cat{A} \xrightarrow{G} \cat{B} \right),
\end{equation}
when the right-hand side exists. If $\cat{A}$ is small, then so is $b{\downarrow}G$, so that the limit in~\eqref{eq:ptw_pf} is guaranteed to exist if $\cat{B}$ is complete. We can relax this condition by transferring the smallness hypothesis from $\cat{A}$ to $G$.

\begin{definition}
    A functor $P \colon \cat{A} \to \Set$ is \demph{small} if it is a small colimit of representables. A functor $G \colon \cat{A} \to \cat{B}$ is \demph{representably small} if $\cat{B}(b,G-) \colon \cat{A} \to \Set$ is small for all $b \in \cat{B}$.
\end{definition}

This terminology is taken from Day and Lack's work in~\cite{Day2007}, where they study limits in categories of small functors. Note that what we call representably small is what Avery and Leinster call corepresentably small in~\cite[Def.~4.6]{Avery2021}.

If a functor $P \colon \cat{A} \to \Set$ is representably small then it is small, since one can take $b$ to be the singleton in the definition. The reverse implication does not hold, however; see~\cite[Ex.~8.1]{Day2007} for a counterexample. The following proposition gives a useful equivalent condition for a functor to be small. We say a category is \demph{cofinally small} if it admits a cofinal functor from a small category.

\begin{proposition}[]\label{prop:small_iff}
    The following are equivalent for a functor $P \colon \cat{A} \to \Set$:
    \begin{enumerate}[(i)]
    \item \label{prop:si:small} $P$ is small;
    \item \label{prop:si:Lan} $P = \Lan_K H$ for some functors $K \colon \cat{B} \to \cat{A}$ and $H \colon \cat{B} \to \Set$ with $\cat{B}$ small;
    \item \label{prop:si:cofinal} the category of elements $\El(P)$ of $P$ is cofinally small.
    \end{enumerate}
\end{proposition}
\begin{proof}
    The equivalence between~\ref{prop:si:small} and~\ref{prop:si:Lan} is shown in a proposition of Kelly~\cite[Prop.~4.83]{Kelly1982}.
    
    Now assume~\ref{prop:si:Lan}, and let $\eta \colon H \to PK$ be the unit of the left Kan extension. Then $\eta$ induces a functor $\El(H) \to \El(PK)$, which we can compose with the obvious functor $\El(PK) \to \El(P)$ to get a functor $D \colon \El(H) \to \El(P)$. It suffices to show that $D$ is cofinal, i.e.\ that the comma category $D{\downarrow}(a,x)$ is (nonempty and) connected for every $a \in \cat{A}$ and $x \in Pa$. Since $\cat{B}$ is small and $\Set$ is cocomplete, the left Kan extension is pointwise, so the assignment
    \[\begin{tikzcd}
    Kb \ar[r, "f"] & a \ar[r, mapsto, shorten <>=15pt] &[2em] Hb \ar[r, "\eta_b"] & PKb \ar[r, "Pf"] & Pa
    \end{tikzcd}\]
    forms a colimiting cocone for $K{\downarrow}a \xrightarrow{\Pi_a} \cat{B} \xrightarrow{H} \Set$. Hence, there exists some $f\colon Kb \to a$ and $y \in Hb$ such that $(Pf \cdot \eta_b) (y) = x$, and so $f$ is a morphism $D(b,y) \to (a,x)$ in $\El(P)$. Moreover, for any other $f' \colon D(b',y') \to (a,x)$, there must be a zig-zag of morphisms in $\El(H\Pi_a)$ connecting $(b,y)$ and $(b',y')$, by the computation of colimits in $\Set$. The same zig-zag connects them in $D{\downarrow}x$, so this category is connected.

    Lastly, we show that~\ref{prop:si:cofinal} implies~\ref{prop:si:small}. One of the consequences of the Yoneda lemma is that $P$ is the colimit of the composite $\El(P)^\opp \to \cat{A}^\opp \to [\cat{A},\Set]$. If $\El(P)$ admits a cofinal functor from a small category, then this colimit is equivalent to a small colimit, showing that $P$ is small.
\end{proof}

\begin{remark}\label{rmk:sub_small}
  Despite what the name might suggest, a subfunctor of a small functor need not be small. Let $\cat{S}$ be a large discrete category, and $\cat{S}_0$ be the category resulting from freely adjoining an initial object, denoted by $0$, to $\cat{S}$. Then the functor $G \colon \cat{S}_0 \to \Set$ which is constant at the singleton set is small (in fact representable), but the subfunctor $F$ of $G$ which sends $0$ to $\emptyset$ and acts as $G$ otherwise is not small. This is easily seen, by the previous proposition, from the fact that $\El(F)$ is isomorphic to $\cat{S}$, which is clearly not cofinally small.
\end{remark}

\begin{examples}\label{ex:rep_small}
    \begin{enumerate}[(i)]
        \item[]
        \item A functor $\cat{A} \to \Set$ from an (essentially) small category is small by taking $K = 1_\cat{A}$ in Proposition~\ref{prop:small_iff}\ref{prop:si:Lan}. Consequently, any functor $\cat{A} \to \cat{B}$ is representably small.
        
        \item \label{ex:rs:adjoint} A right adjoint $G \colon \cat{A} \to \cat{B}$ is representably small, since $\cat{B}(b,G-) \cong \cat{A}(Fb,-)$ is already representable, where $F \dashv G$.
        
        \item \label{ex:rs:discrete} The full and faithful functor $D \colon \Set \to \Top$ that equips a set with the discrete topology is representably small. Indeed, given a topological space $X$, the full subcategory $\cat{S}$ of $X{\downarrow}D = \El(\Top(X,D-))$ on the surjections is cofinal. This follows easily from the facts that $\Top$ has an (epi, strong mono) factorisation system and that subspaces of discrete spaces are discrete. Note that $\cat{S}$ is small, since for any surjection $X \to DS$ we must have $\abs{S} \leq \abs{X}$.
        
        \item \label{ex:rs:field} The full and faithful forgetful functor $F \colon \Field \to \Ring$ is representably small. Given a ring $R$ and a field $k$, any homomorphism $f \colon R \to k$ factors through a unique residue field of $R$, namely $\Frac(R/\p)$, where $\p = \ker f$. Hence, $R{\downarrow}F$ has connected components in bijection with $\Spec R$, and each of those components has an initial object.
        
        \item \label{ex:rs:plus_E} The functor $(-)+E \colon \Set \to \Set$ for $E$ any set is representably small. Denote this functor by $P_E$. Given a set $X$, the connected components of $X{\downarrow}P_E$ are in bijection with the functions $f \colon X \to 1 + E$, and the component corresponding to $f$ has an initial object given by the function $X \to f^{-1}1 + E$ which acts as the identity on $f^{-1}1$ and as $f$ otherwise. 
        
        \item \label{ex:rs:ge_K} Let $K$ be a set, and $\Set_{\geq K}$ be the full subcategory of $\Set$ on the sets of cardinality at least that of $K$. The inclusion $i \colon \Set_{\geq K} \hookrightarrow \Set$ is representably small. Indeed, let $X$ be a set. If $\abs{X} \geq \abs{K}$, then $X{\downarrow}i$ has an initial object. Otherwise, by cardinal comparability, there exists an injection $m \colon X \to K$. We claim that the subcategory $\cat{S}$ of $X{\downarrow}i$ depicted by
        \[\begin{tikzcd}
        X \ar[d, "m"'] \ar[dr, "m'"]\\
        K \ar[r, shift left, "n_1"] \ar[r, shift right, "n_2"'] & K',
        \end{tikzcd}\]
        where $(n_1,n_2)$ is the cokernel pair of $m$, and $m' = n_1m = n_2m$, is cofinal. 
        
        First note that $n_1$ and $n_2$ are also injective, so $K' \in \Set_{\geq K}$. Now take any $f \colon X \to Y$ in $X{\downarrow}i$. It factors through $m$, either because $X = \varnothing$ or because $m$ splits, so $\cat{S}{\downarrow}f$ is nonempty. To see that it is connected, we illustrate the most complicated case: where $f$ factors as $gm'$ and $hm'$. The next diagram shows a zig-zag connecting $g$ and $h$ in $\cat{S}{\downarrow}f$:
        \[\begin{tikzcd}[column sep=large]
        & K' \ar[ddr, "g", bend left] \\
        & K \ar[u, "n_1"] \ar[d, "n_2"'] \\
        X \ar[uur, "m'", bend left] \ar[ur, "m"] \ar[r, "m'"] \ar[dr, "m"'] \ar[ddr, "m'"', bend right]
        & K' \ar[r, "\ang{hn_2, gn_1}"] & Y \\
        & K \ar[d, "n_2"'] \ar[u, "n_1"] \\
        & K' \ar[uur, "h"', bend right]
        \end{tikzcd}\]
        where $\ang{hn_2,gn_1}$ is the unique function $K' \to Y$ whose composites with $n_1$ and $n_2$ are $hn_2$ and $gn_1$, respectively.
    \end{enumerate}
\end{examples}

The importance of representably small functors in the theory of pushforward monads in $\CAT$ is clear from the next proposition.

\begin{proposition}
    Let $G \colon \cat{A} \to \cat{B}$ be a representably small functor into a complete category. Then $G_\# T$ exists for any monad $T$ on $\cat{A}$.
\end{proposition}
\begin{proof}
    It suffices to show that the right Kan extension along $G$ of any functor $F \colon \cat{A} \to \cat{X}$ with $\cat{X}$ complete exists. This follows from the limit formula~\eqref{eq:ptw_Ran}, since $\cat{X}$ is complete and $b{\downarrow}G = \El(\cat{B}(b,G-))$ is cofinally small by assumption.
\end{proof}

It follows that one can take pushforwards along each of the functors listed in Examples~\ref{ex:rep_small}. Representably small functors enjoy many convenient properties. For example, they are closed under composition, as shown by Avery and Leinster~\cite[Lem.~4.7]{Avery2021}. With this result, we can show that the pushforward construction can often be iterated.

\begin{proposition}\label{prop:iterate_pf}
    Let $G \colon \cat{A} \to \cat{B}$ be a representably small functor into a complete category, and $T$ be a monad on $\cat{A}$. If $\cat{A}^T$ has finite connected colimits, then the top functor in the square
    \[\begin{tikzcd}
    \cat{A}^T \ar[r, "G'"] \ar[d, "U^T"'] & \cat{B}^{G_\#T} \ar[d, "U^{G_\#T}"] \\
    \cat{A} \ar[r, "G"'] & \cat{B},
    \end{tikzcd}\]
    corresponding under Proposition~\ref{prop:lax_trans_EM_lift} to the counit of $G_\#T$, is also a representably small functor into a complete category.
\end{proposition}
\begin{proof}
    That $\cat{B}^{G_\#T}$ is complete follows from the fact that $U^{G_\#T}$ creates limits. By Example~\ref{ex:rep_small}\ref{ex:rs:adjoint} and the fact that the composite of two representably small functors is representably small, $GU^T = U^{G_\#T} G'$ is representably small, so $\cat{B}(b, U^{G_\#T}G'-) \cong \cat{B}^{G_\#T}(F^{G_\#T}b, G'-)$ is small for every $b \in \cat{B}$. We can now use the theory of limits in categories of small presheaves to show that $\cat{B}^{G_\#T}(b,G'-)$ is small for every $b \in \cat{B}^{G_\#T}$, not just the free ones.

    Let $[\cat{A}^T,\Set]_\rms$ denote the full subcategory of $[\cat{A}^T,\Set]$ on the small functors. Day and Lack's~\cite[Prop.~4.3]{Day2007} shows that $[\cat{A}^T,\Set]_\rms$ has all finite connected limits and, as the Yoneda embedding $(\cat{A}^T)^\opp \to [\cat{A}^T,\Set]$ factors through $[\cat{A}^T,\Set]_\rms$, it is easy to see that such limits are computed objectwise. Every $b \in \cat{B}^{G_\#T}$ is a reflexive coequaliser of free $G_\#T$-algebras, say of the pair of morphisms $(f,g)$. Since equalisers are finite connected limits, we can take the equaliser $E$ of $\cat{B}^{G_\#T}(f,G'-)$ and $\cat{B}^{G_\#T}(g,G'-)$ in $[\cat{A}^T,\Set]_\rms$, which satisfies
    \[Ea \cong \cat{B}^{G_\#T}(\mathrm{coeq}(f,g), G'a) \cong \cat{B}^{G_\#T}(b, G'a),\]
    naturally in $a \in \cat{A}^T$. Hence, $\cat{B}^{G_\#T}(b,G'-) \cong E$ is small.
\end{proof}

\begin{remark}\label{rmk:sub_rep_small}
  It seems that the hypothesis about the existence of certain colimits in $\cat{A}^T$ cannot be dropped if one wants to conclude that $G'$ is representably small in general. This is because a composite $GF$ being representably small does not imply that $F$ is representably small, even when $G$ is faithful. (A counterexample where $G$ is monadic has not been found.)

  For an example, let $\cat{S}_0$ be as in Remark~\ref{rmk:sub_small}, and let $\mathsf{2}$ denote the category with two objects, $a$ and $b$, and exactly one non-identity morphism $a \to b$. The functor $F \colon \cat{S}_0 \to \mathsf{2}$ which sends $0$ to $a$ and the rest of $\cat{S}$ to $b$ is not representably small, since $\cat{S}_0(b,F-)$ is not small by Remark~\ref{rmk:sub_small}. Now let $G$ be the unique functor $\mathsf{2} \to \mathsf{1}$, which is faithful. As $\mathsf{1}(*,GF-)$ is constant at the singleton (where we have written $*$ for the unique object of $\mathsf{1}$), $GF$ is representably small by the same remark.
\end{remark}

\begin{example}\label{ex:towers}
    Of course, another way of having $G'$ in Proposition~\ref{prop:iterate_pf} be representably small is simply having $\cat{A}$ small. Taking $T$ to be the identity monad, this gives a codensity version of monadic towers. These (or rather their comonadic dual) were introduced by Appelgate and Tierney in~\cite{Appelgate1970}, where they are used to decompose an adjunction into a reflection followed by a number of monadic adjunctions. In the codensity case, we need not start with an adjunction, a functor suffices; and instead of a reflection, the process ends once we reach a \demph{codense} functor, i.e.\ one whose codensity monad is the identity. Here are some examples (see Figure~\ref{fig:towers}):
    \begin{enumerate}[(i)]
        \item Let $2 \colon \mathsf{1} \to \Set$ be the functor which picks out a two-element set. From Example~\ref{ex:pushforward}\ref{ex:pf:terminal}, it follows that the codensity monad of $2$ is the double-powerset monad. It is a classical result that its category of algebras is $\Set^\opp$, which is equivalent to the category of complete atomic Boolean algebras by Stone duality. The functor $2' \colon \mathsf{1} \to \Set^\opp$ picks out $1$, because its composite with $\mathscr{P} \colon \Set^\opp \to \Set$ is $2$. Since $1$ is a dense generator of $\Set$, the functor $2'$ is codense.
        \item Let $i \colon \fdVect \to \Vect$ be the inclusion of the full subcategory of finite dimensional vector spaces over a fixed field. As proved by Leinster~\cite[\S7]{Leinster2013}, its codensity monad is the double dualisation monad, whose category of algebras is equivalent to $\Vect^\opp$. The functor $i' \colon \fdVect \to \Vect^\opp$ is the composite of $i^\opp$ and the equivalence $(-)^* \colon \fdVect \to \fdVect^\opp$ given by the fact that finite dimensional vector spaces are self-dual. Now $i$ is dense, because it is the inclusion of the category of finitely presentable objects, so $i'$ is codense.
        \item \label{ex:towers:Stone} Let $i \colon \FinSet \to \Set$ be the inclusion of the full subcategory of finite sets. Its codensity monad is the ultrafilter monad on $\Set$, whose category of algebras is $\CHaus$. The functor $i' \colon \FinSet \to \CHaus$ equips a finite set with its unique compact Hausdorff topology: the discrete topology. Its codensity monad $i'_\#1$ sends a space $X$ to the set of ultrafilters on its Boolean algebra of clopen subsets. The category of $i'_\#1$-algebras is equivalent to $\Stone$, the category of Stone spaces. This follows from the argument outlined by Sipo\textcommabelow{s} in~\cite[\S5]{Sipos2018}:

        Note that $i'$ factors as $\FinSet \xrightarrow{j} \Stone \xrightarrow{U} \CHaus$. Under Stone duality, $j^\opp$ corresponds to the inclusion of the category of finite Boolean algebras into the category of all Boolean algebras. The former category contains all of the finitely generated free Boolean algebras. Since Boolean algebras are the algebras of a finitary algebraic theory, it follows that $j^\opp$ is dense, and so $j$ is codense. By Remark~\ref{rmk:pf_preserve}, since $U$ is a right adjoint, we have
        \[i'_\#1 = (Uj)_\#1 = U_\# (j_\#1) = U_\# 1.\]
        But $U$ is already monadic (being the inclusion of a reflective subcategory), so the category of algebras of $U_\#1$ is $\Stone$. It also follows that $i'' = j$, which is codense, so the monadic tower stabilises after two steps.
    \end{enumerate}
    \renewcommand\thesubfigure{\roman{subfigure}}
    \begin{figure}
        \centering
        \begin{subfigure}[b]{0.3\textwidth}
        \centering
        \begin{tikzcd}
        & \Set^\opp \ar[d, "\mathscr{P}"] \\
        \mathsf{1} \ar[r, "{2}"'] \ar[ur, "{2'}"] & \Set
        \end{tikzcd}
        \caption{}
        \end{subfigure}
        \begin{subfigure}[b]{0.3\textwidth}
        \centering
        \begin{tikzcd}
        & \Vect^\opp \ar[d, "(-)^*"] \\
        \fdVect \ar[r, "i"'] \ar[ur, "i'"] & \Vect
        \end{tikzcd}
        \caption{}
        \end{subfigure}
        \begin{subfigure}[b]{0.3\textwidth}
        \centering
        \begin{tikzcd}
        & \Stone \ar[d, "U"] \\
        & \CHaus \ar[d, "V"] \\
        \FinSet \ar[r, "i"'] \ar[ur, "i'"] \ar[uur, bend left, "i''"] & \Set
        \end{tikzcd}
        \caption{}
        \end{subfigure}
        \caption{Three examples of codensity monadic towers.}
        \label{fig:towers}
    \end{figure}
\end{example}

\subsection{An adjunction between categories of monads}

We now revisit the adjunctions found at the end of Section~\ref{sec:general}. One of them takes a particularly nice form when one pushes forwards along a full and faithful functor (see Theorem~\ref{thm:mnd_adj}).

Analogously to extensions, we write $\Rift_G F$ and $\Lift_G F$ for the right and left Kan lifts of $F$ through $G$, as in the next lemma.

\begin{lemma}\label{lem:ff_rift}
    Let $\theta$ be a natural isomorphism fitting into a diagram
    \[\begin{tikzcd}[column sep=small]
    & \cat{B} \ar[dr, "G"] \\
    \cat{A} \ar[ur, "H"] \ar[rr, "F"', ""{name=F}] && \cat{C}
    \ar[from=F, to=1-2, Rightarrow, "\theta", shorten >=7pt, shorten <=3pt, pos=0.32]
    \end{tikzcd}\]
    where $G$ is full and faithful. Then $\Rift_G F = (H,\theta)$ and $\Lift_G F = (H,\theta^{-1})$.
\end{lemma}
\begin{proof}
    We prove the statement about the right Kan lift; the other is dual. Note that since $G$ is full and faithful, so is $G_* \colon [\cat{A},\cat{B}] \to [\cat{A},\cat{C}]$. Given any $\alpha \colon F \to GH'$, it follows that there is a unique $\widehat{\alpha} : H \to H'$ such that $G\widehat{\alpha} = \alpha \cdot \theta^{-1}$. This is exactly what it means for $(H,\theta)$ to be the right Kan lift of $F$ through $G$.
\end{proof}

We aim to use the dual of Corollary~\ref{cor:colax_univ}, as given in~\eqref{eq:lift_lax_univ}. Lemma~\ref{lem:ff_rift} allows us to simplify the conditions on the corresponding monads. Let $G \colon \cat{A} \to \cat{B}$ be full and faithful. This lemma, together with the dual of Lemma~\ref{lem:g-det_iff}, implies that any monad on $\cat{A}$ is $G$-opdetermined (see Definition~\ref{def:g-det}). Moreover, a monad $S$ on $\cat{B}$ such that $G^\#S$ exists and has an isomorphism as counit is precisely one that \demph{restricts along $G$}, i.e.\ such that there exist an endofunctor $S'$ of $\cat{A}$ (which inherits a monad structure) and an isomorphism $GS' \cong SG$, since Lemma~\ref{lem:ff_rift} then implies that $G^\#S = S'$.

\begin{examples}\label{ex:mnd_restrict}
    \begin{enumerate}[(i)]
        \item[]
        \item Clearly, the identity monad restricts along any full and faithful functor.
        \item The following monads on $\Set$ restrict along the inclusion $\FinSet \hookrightarrow \Set$: the filter and ultrafilter monads; the $(-) + E$ monad, for a finite set $E$; the $M \times (-)$ monad, for a finite monoid $M$; the powerset monad; and the endomorphism monad of any finite set (see Example~\ref{ex:pushforward}\ref{ex:pf:terminal}), such as the double-powerset monad.
    \end{enumerate}
\end{examples}

\begin{theorem}\label{thm:mnd_adj}
Let $G \colon \cat{A} \to \cat{B}$ be a full and faithful representably small functor into a complete category. There is an adjunction
\[\begin{tikzcd}
\Mnd_\cat{A} \ar[r, shift left=2, "G_\#", ""'{name=RA}] & \Mnd_\cat{B}^{\mathrm{res}G} \ar[l, shift left=2, "G^\#", ""'{name=LA}]
\ar[from=LA, to=RA, phantom, "\dashv", sloped]
\end{tikzcd}\]
where $\Mnd_\cat{B}^{\mathrm{res}G}$ is the full subcategory of $\Mnd_\cat{B}$ on the monads that restrict along $G$. Moreover, the right adjoint $G_\#$ is full and faithful.
\end{theorem}
\begin{proof}
    Since $G$ is representably small and $\cat{B}$ is complete, $G_\#$ is defined on all $\Mnd_\cat{A}$. Since $G$ is full and faithful, the counit of a right Kan extension along it is an isomorphism, showing that $G_\#$ has the right codomain. Similarly, Lemma~\ref{lem:ff_rift} ensures that $G^\#$ is defined on all $\Mnd_\cat{B}^{\mathrm{res}G}$. By Theorem~\ref{thm:univ} and the dual of Theorem~\ref{thm:colax_corr}, there is a bijection
    \[\Mnd_\cat{A}(G^\#S,T) \cong \Lax_G(T,S) \cong \Mnd_\cat{B}(S,G_\#T)\]
    natural in $S \in \Mnd_\cat{B}^{\mathrm{res}G}$ and $T \in \Mnd_\cat{A}$.

    As $G$ is full and faithful, so are $\Ran_G$ and $G_*$, and hence so is $G_\#$. This can also be easily seen from the fact that the counit of the adjunction is an isomorphism. 
\end{proof}

\begin{example}
    The only endofunctor of $\CHaus$ which fixes underlying sets is the identity. Such an endofunctor amounts to an endomorphism of the ultrafilter monad on $\Set$, by the correspondence between monad maps and functors between Eilenberg--Moore categories. Since $i_\# \colon \Mnd_\FinSet \to \Mnd_\Set$ is full and faithful, this gives and endomorphism of the identity monad on $\FinSet$, which can only be the identity.
\end{example}

\begin{example}
    The adjunction of Theorem~\ref{thm:mnd_adj} takes a particularly simple form when we take $G$ to be $i \colon \Set_{\geq K} \hookrightarrow \Set$ from Example~\ref{ex:rep_small}\ref{ex:rs:ge_K}. Most monads on $\Set$ have a monic unit; these are called the \demph{consistent monads}. The only two inconsistent monads are the one that is constant at $1$, and that which is constant at $1$ except that it sends $\varnothing$ to $\varnothing$ (see~\cite[Lem.~IV.3]{Adamek2012}).
    
    Every consistent monad $T$ restricts along $i$. Moreover, for every set $X$ such that $TX$ is nonempty, we have $(i_\# i^\# T)X \cong TX$. This is clear if $\abs{X} \geq \abs{K}$. Otherwise, Example~\ref{ex:rep_small}\ref{ex:rs:ge_K} shows that $(i_\# i^\# T)X$ is the equaliser of
    \[\begin{tikzcd}
    TK \ar[r, shift left, "Tn_1"] \ar[r, shift right, "Tn_2"'] & TK',
    \end{tikzcd}\]
    where $(n_1,n_2)$ is the cokernel pair of an injection $m \colon X \to K$. Since $TX$ is nonempty, there is a retraction $r$ of $Tm$. We also have the composite
    \[\begin{tikzcd}
    s = TK' \ar[r, "{T\ang{\eta^T_K, Tm \cdot r \cdot \eta^T_K}}"] &[5em] T^2K \ar[r, "\mu^T_K"] & TK,
    \end{tikzcd}\]
    where $\ang{f_1,f_2}$ denotes the unique morphism such that $\ang{f_1,f_2} \cdot n_i = f_i$ for $i \in \{1,2\}$, for $f_1$ and $f_2$ such that $f_1m = f_2m$. One checks that $r$ and $s$, together with $Tm$, $Tn_1$ and $Tn_2$ form a split equaliser diagram, showing that $(i_\#i^\#T)X \cong TX$.  

    If we take $K = 1$, then even the inconsistent monads restrict along $i$, giving a reflection
    \begin{equation}\label{eq:mnd_refl}
    \begin{tikzcd}
    \Mnd_{\Set_{\geq 1}} \ar[r, hook, shift left=2, "i_\#", ""'{name=RA}] & \Mnd_\Set. \ar[l, shift left=2, "i^\#", ""'{name=LA}]
    \ar[from=LA, to=RA, phantom, "\dashv", sloped]
    \end{tikzcd}
    \end{equation}

    Let $a$ and $b$ be the two functions $1 \to 2$. Given a monad $T$ on $\Set$, the equaliser of the pair $(Ta,Tb)$ is the set of \demph{pseudoconstants} of $T$. A pseudoconstant can be understood as a unary term in the theory of $T$ which takes a constant value in every nonempty algebra. Every constant induces a pseudoconstant; an example of a pseudoconstant that does not come from a constant is the unique element of any nonempty model of the theory of sets where all elements are equal. 
    
    The monad on $\Mnd_\Set$ induced by the reflection~\eqref{eq:mnd_refl} sends a monad $T$ to $i_\#i^\#T$, which agrees with $T$ on all nonempty sets and which sends $\varnothing$ to the set of pseudoconstants of $T$. In other words, $i_\#i^\#$ realises the pseudoconstants of a monad as actual constants. The argument above implies the well-known fact that as soon as there is at least one constant, all pseudoconstants are induced by constants.
\end{example}

As outlined at the end of Subsection~\ref{ssec:g-det}, there is also a dual partial adjunction (notice the interchange between the left and right adjoints):
\[\begin{tikzcd}
\Mnd_\cat{A} \ar[r, shift left=2, "G_\#", ""'{name=LA}] & \Mnd_\cat{B}^{G\mathrm{det}} \ar[l, shift left=2, "G^\#", ""'{name=RA}]
\ar[from=LA, to=RA, phantom, "\dashv", sloped]
\end{tikzcd}\]
where $\Mnd_\cat{B}^{G\mathrm{det}}$ is the full subcategory of $\Mnd_\cat{B}$ on the $G$-determined monads. In this case, $G^\#$ need not be defined on all of $\Mnd_\cat{B}^{G\mathrm{det}}$. A simple sufficient condition to make this adjunction total is that $G$ be a full and faithful left adjoint (such as the discrete-topology functor $\Set \to \Top$), since then right Kan lifts through $G$ always exist.

\begin{examples}\label{ex:G-det}
    \begin{enumerate}[(i)]
        \item[] 
        \item If $G \colon \cat{A} \to \cat{B}$ is full and faithful, then any right Kan extension along $G$ is $G$-determined by Lemma~\ref{lem:g-det_iff}. For example, the endomorphism monad on $\Set$ of a finite set $X$ is $i$-determined, where $i \colon \FinSet \to \Set$. This follows from Example~\ref{ex:pushforward}\ref{ex:pf:terminal}, because it is the codensity monad of $\mathsf{1} \xrightarrow{X} \FinSet \xrightarrow{i} \Set$, and hence is given by $\Ran_{iX} iX = \Ran_i (\Ran_X iX)$. 
        
        In particular, the double-powerset monad on $\Set$ is $i$-determined. This is in stark contrast to the powerset monad $\PS$ on $\Set$, which is not $i$-determined. If it were, then $\PS = \Ran_i \PS i$, but we will see in Section~\ref{sec:finset} that the right-hand side is the filter monad.
        \item Let $i \colon \cat{A} \hookrightarrow \cat{B}$ be the inclusion of a full subcategory. By definition, a monad on $\cat{B}$ is $i$-determined precisely when it has the \textit{limit property} of Ad\'amek and Sousa's~\cite[Def.~6.2]{Adamek2021}. They give a series of examples of double dualisation monads on complete symmetric monoidal categories $\cat{C}$ which are $i$-determined for $i \colon \cat{C}_\mathsf{fp} \hookrightarrow \cat{C}$ the inclusion of the finitely presentable objects. 
        
        The main idea of Theorem~6.5 of~\cite{Adamek2021} can be generalised easily using our theory of pushforwards as the following statement: Let $G \colon \cat{A} \to \cat{B}$ be a functor such that $G_\#1$ exists. If $T$ is a $G$-determined monad on $\cat{B}$ such that $\eta^T G$ is a a monic natural transformation, then $G_\#1$ is the smallest $G$-determined submonad of $T$.

        This follows from the fact that for any monad $T$ on $\cat{B}$, the pair $(G, \eta^T G)$ is a colax transformation $1_\cat{A} \to T$. If $T$ is $G$-determined and $\eta^T G$ is monic, then Theorem~\ref{thm:colax_corr} gives a monic monad map $\widehat{\eta^T G} \colon G_\#1 \to T$. If $\theta \colon S \to T$ is a $G$-determined submonad, then the naturality in $T$ of this construction ensures that $\theta \cdot \widehat{\eta^S G} = \widehat{\eta^T G}$.
    \end{enumerate}
\end{examples}

\subsection{Limit completions and codensity}\label{ssec:codense}

We finish this section with the observation that codensity monads are invariant under limit completions, in a suitable sense. This allows us to relate codensity monads to Diers's multiadjunctions~\cite[p.~58]{Diers1981} and Tholen's $\mathfrak{D}$-pro-adjunctions~\cite[p.~148]{Tholen1984}.

Let $G \colon \cat{A} \to \cat{B}$ be a functor with $\cat{A}$ essentially small and $\cat{B}$ complete. We get a usual nerve-realisation-type adjunction:
\[\begin{tikzcd}[column sep=small]
& \cat{A} \ar[ddl, "y"', end anchor=north, hook'] \ar[ddr, "G", end anchor=north] &[1.5em] \\ \\
{[\cat{A}, \Set]^\opp} \ar[rr, shift right=2, ""{below, name=R}, "\Ran_y G"'] && \cat{B} \ar[ll, shift right=2, ""{above, name=L}, "\Ran_G y"']
\ar[phantom, sloped, "\dashv", from=L, to=R]
\end{tikzcd}\]
If we think of $[\cat{A}, \Set]^\opp$ as the free limit completion of $\cat{A}$, then the right adjoint realises the formal limit of a small diagram $D \colon \cat{I} \to \cat{A}$ as the actual limit in $\cat{B}$ of $GD$. As Leinster explains in~\cite[\S 2]{Leinster2013}, the monad on $\cat{B}$ that this adjunction induces is precisely the codensity monad of $G$. A key aspect in this situation is that the Yoneda embedding $y$ is codense.

Recall that a functor is codense when its codensity monad is the identity. Codense functors play an important role in relating the codensity monads of different functors, as shown by the next proposition.

\begin{proposition}\label{prop:codense_preserve_codensity}
    Let $G \colon \cat{A} \to \cat{B}$ and $H \colon \cat{X} \to \cat{A}$ be functors with $H$ codense. If $G$ is a right adjoint, then $G$ and $GH$ have the same codensity monad.
\end{proposition}
\begin{proof}
    We have
    \[(GH)_\# 1 \cong G_\# (H_\# 1) \cong G_\# 1,\]
    where the first isomorphism is that of Lemma~\ref{lem:pf_preserve}, since $G$ is a right adjoint, and the second is the fact that $H$ is codense.
\end{proof}

\begin{remark}
    The condition that $G$ be a right adjoint cannot be dropped. The following example is due to Kelly~\cite[\S5.2]{Kelly1982}: the functors $1 \colon \mathsf{1} \to \FinSet$ and $y \colon \FinSet \to [\FinSet^\opp, \Set]$ are both dense, but their composite is not. It follows that the codensity monads of $y^\opp$ and of $y^\opp 1^\opp$ are different. 
\end{remark}

The next lemma is a rich source of codense functors.

\begin{lemma}[{Kelly~\cite[Thm.~5.13]{Kelly1982}}]\label{lem:codense_functor}
    Let \begin{tikzcd}[cramped, sep=scriptsize]
        \cat{A} \ar[r, "F"] & \cat{B} \ar[r, "G"] & \cat{C}
    \end{tikzcd} be functors, with $G$ full and faithful. If $GF$ is codense, then so are $F$ and $G$.
\end{lemma}

For any category $\cat{A}$, the category $[\cat{A},\Set]^\opp_\rms$ is its free small-limit completion, and the Yoneda embedding $y \colon \cat{A} \to [\cat{A},\Set]^\opp_\rms$ is codense. It follows that the inclusion of $\cat{A}$ into any full subcategory of $[\cat{A},\Set]^\opp_\rms$ containing the representables is codense. Such subcategories can be thought of as the categories obtained from $\cat{A}$ by freely adjoining limits for a chosen class of diagrams.

These facts together with Proposition~\ref{prop:codense_preserve_codensity} allow us to relate codensity monads to the multiadjunctions of Diers~\cite{Diers1980, Diers1981} and, more generally, the pro-adjunctions of Tholen~\cite{Tholen1984}. We summarise the situation now.

Let $\mathfrak{D}$ be a class of small categories containing the terminal category $\mathsf{1}$. Let $\Pro(\mathfrak{D}, \cat{A})$ be the full subcategory of $[\cat{A},\Set]^\opp_\rms$ on those objects that are $\cat{I}^\opp$-indexed limits of representables, for $\cat{I} \in \mathfrak{D}$. Since $\mathsf{1} \in \mathfrak{D}$, the Yoneda embedding factors as a fully faithful functor $\cat{A} \hookrightarrow \Pro(\mathfrak{D}, \cat{A})$. If $\cat{A}$ is \textbf{$\mathfrak{D}$-complete}, i.e.\ it has $\cat{I}^\opp$-indexed limits for all $\cat{I} \in \frak{D}$, then this inclusion has a right adjoint given by taking the corresponding limit in $\cat{A}$. Any functor $G \colon \cat{A} \to \cat{B}$ induces a functor $\Pro(\mathfrak{D}, G)$ making the diagram
\[\begin{tikzcd}
\cat{A} \ar[r, "G"] \ar[d, hook] &[2em] \cat{B} \ar[d, hook] \\
\Pro(\frak{D}, \cat{A}) \ar[r, "\Pro{(\frak{D}, G)}"] & \Pro(\frak{D}, \cat{B})
\end{tikzcd}\]
commute. 

\begin{definition}[{Tholen~\cite{Tholen1984}}]\label{def:pro_adj_mon}
    A functor $G \colon \cat{A} \to \cat{B}$ is
    \begin{enumerate}[(i)]
        \item a \textbf{right $\mathfrak{D}$-pro-adjoint} if $\Pro(\frak{D}, G)$ is a right adjoint;
        \item \textbf{$\frak{D}$-pro-monadic} if $\Pro(\frak{D},G)$ is monadic.
    \end{enumerate}
    Let $\frak{S}$ be the class of all sets (small discrete categories). Then $G$ is a \textbf{right multiadjoint} if it is a right $\frak{S}$-pro-adjoint, and it is \textbf{multimonadic} if it is $\frak{S}$-pro-monadic.
\end{definition}

\begin{proposition}\label{prop:codensity_pro_adjoint}
    If $G \colon \cat{A} \to \cat{B}$ is a right $\frak{D}$-pro-adjoint and $\cat{B}$ is $\frak{D}$-complete, then $G$ has a codensity monad and it is the monad induced by the right adjoint
    \begin{equation}\label{eq:pro_adjoint}
        \begin{tikzcd}
            \Pro(\frak{D}, \cat{A}) \ar[r, "\Pro{(\frak{D}, G)}"] &[2em] \Pro(\frak{D}, \cat{B}) \ar[r] & \cat{B}.
        \end{tikzcd}
    \end{equation}
\end{proposition}
\begin{proof}
    Since $\Pro(\mathfrak{D}, \cat{A}) \hookrightarrow [\cat{A},\Set]^\opp_\rms$ is fully faithful, the inclusion $\cat{A} \hookrightarrow \Pro(\mathfrak{D}, \cat{A})$ is codense by Lemma~\ref{lem:codense_functor}. The result then follows from Proposition~\ref{prop:codense_preserve_codensity}.
\end{proof}

The case of multiadjoints and multimonadic functors was studied by Diers~\cite{Diers1980,Diers1981a, Diers1981}. In this case, the category $\Pro(\frak{S}, \cat{A})$ is the free product completion of $\cat{A}$, which we denote by $\Prod(\cat{A})$.

\begin{lemma}[Tholen]\label{lem:multiadjoint}
    A functor $G \colon \cat{A} \to \cat{B}$ is a right multiadjoint iff, for each $b \in \cat{B}$, the category $b{\downarrow}G$ has a set of connected components and each component has an initial object.
\end{lemma}
\begin{proof}
    The second condition is the original definition of right multiadjoint given by Diers~\cite[p.~58]{Diers1981}. The equivalence with the definition here was proved by Tholen in~\cite[Thm.~2.4]{Tholen1984}.
\end{proof}

There are analogues of the monadicity theorems for multimonadic functors, many of which can be found in Diers~\cite[\S 3 and~4]{Diers1980}. However, there is a sufficient condition for multimonadicity that is in practice easy to check.

\begin{definition}\label{def:rel_ff}
    A functor $G \colon \cat{A} \to \cat{B}$ is \textbf{relatively full and faithful} if every morphism of $\cat{A}$ is $G$-cartesian, i.e.\ for any pair of morphisms $f \colon X \to Z$ and $g \colon Y \to Z$ in $\cat{A}$ with the same codomain, and any $m \colon GX \to GY$ in $\cat{B}$ such that $Gg \cdot m = Gf$, there exists a unique $h \colon X \to Y$ such that $hg = f$ and $Gh = m$.
\end{definition}

This conditions may be summarised by the following diagram.
\[\begin{tikzcd}[row sep=tiny, column sep=scriptsize]
X \ar[dd, "\exists! h"'] \ar[dr, near start, "\forall f"] & &[-2em] & &[-2em] GX \ar[dd, "\forall m"'] \ar[dr, near start, "Gf"] \\
& Z & \phantom{1} \ar[r, mapsto, "G"] & \phantom{1} & & GZ \\
Y \ar[ur, near start, "\forall g"'] & & & & GY \ar[ur, near start, "Gg"']
\end{tikzcd}\]

The next proposition is part of Diers~\cite[Prop.~6.0]{Diers1980}.

\begin{proposition}[Diers]\label{prop:multiadjoint_rel_ff}
    A functor that is both a right multiadjoint and relatively full and faithful is multimonadic.
\end{proposition}

For ease of reference, we reproduce one last powerful result about multimonadic functors which can be found in Diers~\cite[p.~661]{Diers1981a}. It relates the monadicity of the functor~\eqref{eq:pro_adjoint} to the multimonadicity of $G$, at least in the case of categories over $\Set$.

\begin{theorem}[Diers]\label{thm:multimonadic_iff}
    Let $G \colon \cat{A} \to \Set$ be a functor. The functor
    \[\begin{tikzcd}
    \Prod(\cat{A}) \ar[r, "\Prod(G)"] &[2em] \Prod(\Set) \ar[r] & \Set
    \end{tikzcd}\]
    is monadic iff $G$ is multimonadic and it satisfies the following condition:
    \begin{equation}\label{eq:no_prod_cond} \tag{D}
        \begin{tabular}{cc}
            \text{for any set $I$, $A \in \cat{A}$, and any family of morphisms $(f_i \colon A \to A_i)_{i \in I}$ in $\cat{A}$,} \\ 
            \text{if $(GA, Gf_i)_{i \in I}$ is a product cone, then $I$ has exactly one element.}
        \end{tabular}
    \end{equation}
    If this is the case, the corresponding monad is the codensity monad of $G$.
\end{theorem}

Condition~\eqref{eq:no_prod_cond} implies that $G$ does not create any products indexed by a set $I$, unless $I$ is a singleton. If $G$ reflects products (e.g.\ if $G$ is full and faithful), then these two conditions are equivalent.

\begin{examples}
    \begin{enumerate}[(i)]
        \item[]
        \item The argument in Example~\ref{ex:rep_small}\ref{ex:rs:field} shows, in the light of Lemma~\ref{lem:multiadjoint}, that the forgetful functor $F \colon \Field \to \Ring$ is a right multiadjoint. It is clear from Definition~\ref{def:pro_adj_mon} that right multiadjoints are closed under composition, so the composite $UF \colon \Field \to \Ring \to \Set$ is also a right multiadjoint. Since all field homomorphisms are injective, it is easy to see that $UF$ is relatively full and faithful, and hence multimonadic by Proposition~\ref{prop:multiadjoint_rel_ff}. Moreover, $UF$ reflects limits, because both $U$ and $F$ do, and it is clear that it does not create any nontrivial products. It follows from Theorem~\ref{thm:multimonadic_iff} that the functor $\Prod(\Field) \to \Set$ taking a formal product of fields to the product of their underlying sets is monadic, and that it is the monadic reflection of $UF$.
        
        This functor factors as $\Prod(\Field) \xrightarrow{P} \Ring \xrightarrow{U} \Set$, where $U$ is monadic and $P$ is a right adjoint. It follows from Beck's monadicity theorem that $P$ is also monadic, and from Proposition~\ref{prop:codensity_pro_adjoint} that the corresponding monad is the codensity monad of $F$.
        \item Similarly, Example~\ref{ex:rep_small}\ref{ex:rs:plus_E} shows that $P_E \colon \Set \to \Set$ is a right multiadjoint. In fact, it is also relatively full and faithful. Indeed, let $f \colon X \to Z$ and $g \colon Y \to Z$ be functions, and $m \colon P_EX \to P_EY$ be such that $P_Eg \cdot m = P_E f$. It follows that $m$ must act as the identity on $E$, so it is of the form $P_Eh$ for a unique $h \colon X \to Y$ such that $gh = f$.
        
        Condition~\eqref{eq:no_prod_cond} is satisfied as soon as $E$ has at least two elements. It is clear that it fails for $E = \varnothing$ and $1$ (the latter one by taking $I = \varnothing$). Now suppose $e,e' \in E$ are two distinct elements, and let $I$ and $X$ be sets, and $(f_i \colon X \to X_i)_{i \in I}$ be a cone such that $(P_EX,P_Ef_i)_{i \in I}$ is a product cone. Since $P_EX$ is not terminal, $I$ is nonempty. If $I$ has at least two elements, $i$ and $i'$, then there is some element of $P_EX$ that maps to $e$ under $P_Ef_i$ and to $e'$ under $P_Ef_{i'}$, but this is impossible from the definition of $P_E$. It follows that $I$ is a singleton. 

        Theorem~\ref{thm:multimonadic_iff} then gives a proper class of monadic functors $Q_E \colon \Prod(\Set) \to \Set$ (one for each $E$ with at least two elements), sending the formal product of $(X_i)_{i \in I}$ to the cartesian product $\prod_{i \in I} (X_i + E)$.
    \end{enumerate}

    \begin{remark}
        Given that Proposition~\ref{prop:codensity_pro_adjoint} applies to more general limit completions than those under products, it is possible to imagine generalisations of Theorem~\ref{thm:multimonadic_iff} to $\mathfrak{D}$-pro-monadic functors, where $\mathfrak{D}$ is perhaps a sound doctrine of limits. We leave this investigation for future work.
    \end{remark}
\end{examples}

\section{Pushing forward monads on finite sets}
\label{sec:finset}

In this last section, we explicitly compute the pushforward along $i \colon \FinSet \hookrightarrow \Set$ of several familiar monads on $\FinSet$. We will also show that, with the exception of two examples, all pushforwards along $i$ have no rank.

Note that $i$ is a full and faithful, representably small functor into a complete category, so we are in the setting of Theorem~\ref{thm:mnd_adj}. In particular, the problem of finding monads on $\FinSet$ reduces to finding monads on $\Set$ that restrict along $i$, i.e.\ that preserve finiteness.

\begin{examples}\label{ex:easy_pf_finset}
    There are at least three (families of) monads whose pushforward along $i$ is easy to identify.
    \begin{enumerate}[(i)]
        \item \label{ex:pff:T} Let $T$ be the terminal monad on $\FinSet$, i.e.\ the one that is constant at $1$. It follows from the limit preservation property of Lemma~\ref{lem:pf_functor} that $i_\#T$ is the terminal monad on $\Set$.
        \item \label{ex:pff:S} Let $S$ be the unique submonad of $T$ such that $S\varnothing = \varnothing$. For any nonempty set $X$, the limit formula in~\eqref{eq:ptw_pf} is the same for $i_\#T(X)$ and $i_\#S(X)$, so we have $i_\#S(X) = 1$. Since $i$ is full and faithful, $i_\#S(\varnothing) = \varnothing$. Hence, $i_\#S$ is the other inconsistent monad on $\Set$.
        \item \label{ex:pff:dd} Let $n$ be a finite set, and $D$ be the endomorphism monad of $n$ in $\FinSet$, so that $D(X) = n^{\Set(X,n)}$ for any finite set $X$. This is the codensity monad of $n \colon \mathsf{1} \to \FinSet$. Note that the right Kan extension $n_\#1$ is pointwise, and its computation only involves finite limits. Since $i$ preserves finite limits, it preserves $\Ran_n n$, and by Lemma~\ref{lem:pf_preserve} we have $i_\#D = i_\#(n_\#1) \cong (in)_\#1$, which is the endomorphism monad of $n$ in $\Set$.

        The same argument shows that, for any functor $G \colon \cat{A} \to \FinSet$ from a finite category, and any monad $T$ on $\cat{A}$, we have $i_\#(G_\#T) = (iG)_\# T$.
    \end{enumerate}
\end{examples}

We will consider the pushforward of three more families of monads on $\FinSet$, which we introduce now. 

\begin{itemize}
    \item For any finite set $E$, the functor $(\cdot) + E \colon \Set \to \Set$ has a well-known monad structure whose category of algebras is $E/\Set$. We denote this monad by $P_E$ and refer to it as the exception monad. Since it sends finite sets to finite sets, it restricts to a monad on $\FinSet$, which we denote by $P_E^\rmf \coloneqq i^\# P_E$.
    \item For any finite monoid $M$, the functor $M \times (\cdot) \colon \Set \to \Set$ has a well-known monad structure whose algebras are left $M$-sets, i.e.\ sets with an action of $M$ on the left. We denote this monad by $A_M$. Again, since this sends finite sets to finite sets, it restricts to a monad on $\FinSet$, which we denote by $A_M^\rmf \coloneqq i^\# A_M$.
    \item The covariant powerset functor $\PS \colon \Set \to \Set$ has a well-known monad structure whose category of algebras is the category of suplattices. Once again, this sends finite sets to finite sets, so it restricts to a monad on $\FinSet$, which we denote by $\PS^\rmf \coloneqq i^\# \PS$. Algebras for $\PS^\rmf$ are finite bounded join-semilattices, i.e.\ commutative monoids with an idempotent operation ($x^2 = x$ for all $x$).
\end{itemize}

Before delving into these examples, it will be useful to recall the situation for the pushforward of the identity on $\FinSet$, i.e.\ the codensity monad of $i$. Kennison and Gildenhuys~\cite[p.~341]{Kennison1971} showed that $i_\# 1$ is the ultrafilter monad, which we denote by $\beta$, whose category of algebras is the category of compact Hausdorff topological spaces and continuous maps. The functoriality of $i_\#$ and Example~\ref{ex:monad_morphisms}\ref{ex:mm:eta} give monad maps $\beta \to i_\# T^\rmf$ for any monad $T^\rmf$ on $\FinSet$. Moreover, each of the monads described above is the monad lift through $i$ of a monad $T$ on $\Set$. We therefore have a map $T \to i_\#i^\# T$ which is the unit of the adjunction of Theorem~\ref{thm:mnd_adj}. In terms of Eilenberg--Moore categories, this gives a commutative square of forgetful (in fact, monadic) functors
\begin{equation}\label{eq:monadic_square}
\begin{tikzcd}
& \Set^{i_\#i^\#T} \ar[dl, end anchor=north east] \ar[dr, end anchor=north west] \\
\CHaus \ar[dr, "U^\beta"', start anchor=south east, end anchor=north west] && \Set^T \ar[dl, "U^T", start anchor=south west, end anchor=north east] \\
& \Set
\end{tikzcd}
\end{equation}
for any monad $T$ on $\Set$ which restricts along $i$.

This situation highlights a similarity between pushforwards of lifted monads and distributive laws between $T$ and $\beta$, whereby the algebras for $i_\# i^\# T$ have underlying $\beta$- and $T$-algebra structures. In the first two cases, we will see that this similarity is manifest, in that $i_\# P_E^\rmf$ and $i_\# A_M^\rmf$ are the composite monads given by distributive laws between $P_E$ and $\beta$, and $A_M$ and $\beta$, respectively. In the third case, there is no known distributive law between $\PS$ and $\beta$, but there is a \emph{weak} distributive law of $\PS$ over $\beta$, studied by Garner~\cite{Garner2020}. It will turn out that $i_\# \PS^\rmf$ is the composite monad induced by this weak distributive law.

\begin{remark}
  As far as we can tell, the appearance of these (weak) distributive laws is a coincidence. Of course, since the pushforward construction makes sense in any $2$-category, one can think about the existence of pushforwards in $\MND(\CAT)$ or $\EM(\CAT)$. Monads in $\MND(\CAT)$ were identified by Street as distributive laws~\cite[\S6]{Street1972}, while monads in $\EM(\CAT)$ are Lack and Street's concept of monad wreaths~\cite[\S3]{Lack2002}. However, these $2$-categories do not have the usual calculus of weighted limits, which simplifies the process of computing pushforwards.

  Let $S$ and $T$ be monads on a category $\cat{A}$. It is a standard result (see \cite[\S9.2.2]{Barr1985}) that distributive laws $\delta \colon TS \to ST$ correspond to liftings of $S$ to $\cat{A}^T$, by which we mean a monad $S^T$ on $\cat{A}^T$ which satisfies $U^T S^T = S U^T$ and two more axioms. Lack and Street~\cite[p.~257]{Lack2002} showed that a monad wreath amounts to a similar lifting, but this time without the last two axioms. Let $R$ be a monad on $\cat{B}$, and $(G,\phi)$ be a lax transformation from $T$ to $R$, i.e.\ a $1$-cell in $\MND(\CAT)$ or, equivalently, in $\EM(\CAT)$. By Proposition~\eqref{prop:lax_trans_EM_lift}, this corresponds to a commutative square of functors
  \[\begin{tikzcd}
    \cat{A}^T \ar[d, "U^T"'] \ar[r, "G^\phi"]& \cat{B}^R \ar[d, "U^R"]\\
    \cat{A} \ar[r, "G"'] & \cat{B}.
  \end{tikzcd}\]
  Naivelly, one would hope to push the lifted monad $S^T$ on $\cat{A}^T$ forward along $G^\phi$ to get a monad on $\cat{B}^R$ which is a lifting of $G_\#S$, thus pushing forward a distributive law/wreath on $\cat{A}$ to one on $\cat{B}$. However, even the most simple examples show that this is not the case in general. In particular, the equation $U^R (G^\phi_\# S^T) = (G_\#S) U^R$ seems to fail often.

  For instance, one can take $\cat{A} = \FinSet$ and $\cat{B} = \Set$, with $T = S = 1$ and $R = \beta$, and $G = i \colon \FinSet \to \Set$ with $\phi = \eps$, the counit of $i_\# 1 = \beta$. The pushforward $i^\eps_\# 1$ is the monad described in Example~\ref{ex:towers}\ref{ex:towers:Stone} which sends a compact Hausdorff space to the space of utrafilters on its Boolean algebra of clopen subsets. Clearly, $U^\beta (i^\eps_\# 1) \neq (i_\# 1) U^\beta$, as can be seen by evaluating at any connected compact Hausdorff space with an infinite underlying set. The situation is no better when $G$ and $G^\phi$ are right adjoints; consider, for example, $\cat{A} = \Ring$ with $T = S = 1$, and $\cat{B} = \Set$ with $\cat{B}^R = \Ab$.
\end{remark}

Because filters and ultrafilters on a set will play a key role in these examples of pushforwards along $i \colon \FinSet \hookrightarrow \Set$, we now write down the definitions to fix our notation.

\begin{definition}\label{def:filters}
    Let $X$ be a set. A \textbf{filter} on $X$ is a subset $\calF$ of $\PS X$, such that
    \begin{enumerate}[(i)]
        \item $X \in \calF$,
        \item for $A, B \subseteq X$, we have $A \cap B \in \calF$ iff $A, B \in \calF$\footnote{This condition is often split into two: if $A \subseteq B \subseteq X$ and $A \in \calF$, then $B \in \calF$; and if $A, B \in \calF$, then $A \cap B \in \calF$. The equivalence between the two definitions is straightforward.}.
    \end{enumerate}
    An \textbf{ultrafilter} is a filter that also satisfies: for $A \subseteq X$, either $A \in \calF$ or $X \setminus A \in \calF$, but not both. Ordering the filters on $X$ by inclusion, ultrafilters are precisely the maximal proper ones. 
    
    Given $A \subseteq X$, the \textbf{principal filter} at $A$ is the filter
    \[\pf{A} = \{B \subseteq X \mid A \subseteq B \}.\]
    This is an ultrafilter iff $A = \{x\}$ for some $x \in X$, in which case it is called the \textbf{principal ultrafilter} at $x$. Any filter on a finite set is principal; indeed, a filter is closed under finite intersections, and taking the intersection over the entire filter gives a least element.
    
    Let $FX$ and $\beta X$ denote the set of filters and ultrafilters on $X$, respectively. Then $F$ becomes a functor as follows: given a function $f \colon X \to Y$ and $\calF \in FX$, let
    \[Ff(\calF) = \{ B \subseteq Y \mid f^{-1}B \in \calF \}.\]
    One readily checks that this is a filter on $Y$, and that this assignment is functorial. Note that for $A \subseteq X$, we have $Ff(\pf A) = \pf f(A)$. Moreover, $Ff(\calF)$ is an ultrafilter if $\calF$ is, so $\beta$ becomes a subfunctor of $F$.

    Both $F$ and $\beta$ have well-known monad structures. Let $X$ be a set. In both cases, the unit sends $x \in X$ to $\pf{\{x\}}$. Given $\mathfrak{F} \in FFX$, we have
    \begin{equation}\label{eq:filter_mult}
        \mu^F_X(\mathfrak{F}) = \{ A \subseteq X \mid A^\# \in \mathfrak{F}\} \quad \text{where} \quad A^\# = \{\calF \in FX \mid A \in \calF\}.
    \end{equation}
    If $\mathfrak{F} \in \beta \beta X$, then 
    \begin{equation}\label{eq:ultrafilter_mult}
        \mu^\beta_X(\mathfrak{F}) = \{ A \subseteq X \mid A^\# \cap \beta X \in \mathfrak{F} \}.
    \end{equation}
\end{definition}

Since $\FinSet$ is essentially small and $\Set$ is complete, all pushforwards along $i$ are pointwise, given by the limit formula~\eqref{eq:ptw_pf}. For $g \colon X \to Y$ a function into a finite set, we will write $\ang{g} \colon i_\#T^\rmf(X) \to T^\rmf Y$ for the leg of the limit cone at $g \in X{\downarrow}i$. Which monad $T^\rmf$ on $\FinSet$ this refers to will be clear from the context. The cone condition implies that if $f \colon Y \to Y'$ is a function between finite sets then $\ang{fg} = T^\rmf f \cdot \ang{g}$.

As a warmup for the remaining examples of pushforwards, we review how $\beta$ is the codensity monad of $i$. Since every ultrafilter on a finite set is principal, it is easy to see that $i^\# \beta$ is the identity monad, thus giving the unit map $\nu \colon \beta \to i_\# 1$. For a set $X$, an ultrafilter $\calF$ on $X$, and a function $g \colon X \to Y$ into a finite set, we have that $(\ang{g} \cdot \nu_X)(\calF)$ is the element at which $\beta g (\calF)$ is principal, i.e.\ the unique $y \in Y$ such that $\{y\} \in \beta g(\calF)$, or equivalently such that $g^{-1}\{y\} \in \calF$. The inverse of $\nu_X$ sends $\phi \in i_\#1(X)$ to the set of those subsets $A \subseteq X$ such that $\ang{\chi_A} (\phi) = \top$, where $\chi_A \colon X \to 2 = \{\bot,\top\}$ is the characteristic function of $A$. Note that $\nu_X (\pf{\{x\}})$ is given by evaluation at $x$, in the sense that $(\ang{g} \cdot \nu_X)(\pf{\{x\}}) = g(x)$.

\subsection{The exception and \texorpdfstring{$M$}{M}-set monads}
\label{ssec:excp_Mset}

The cases of $P^\rmf_E = (\cdot) + E$ and $A^\rmf_M = M \times (\cdot)$ are similar, so we treat them together. We will need the following standard fact about the ultrafilter monad.
\begin{proposition}\label{prop:ultra_preserve_coprod}
    The functor $\beta$ preserves finite coproducts. In particular, for any set $X$, and finite sets $E$ and $M$, we have isomorphisms
    \[\beta(X + E) \cong \beta X + E \quad \text{and} \quad \beta (M \times X) \cong M \times \beta X\]
    natural in $X$. 
\end{proposition}
\begin{proof}
    Let $X$ and $Y$ be sets. For $\calF \in \beta (X + Y)$, either $X \in \calF$ or $Y \in \calF$, but not both. If $X \in \calF$, then $\calF \cap \PS X$ is an ultrafilter on $X$, and similarly for $Y$. This gives an inverse to the canonical map $\beta X + \beta Y \to \beta (X + Y)$.
\end{proof}

In fact, these isomorphisms turn out to be distributive laws of $\beta $ over $P_E$ and $A_M$, respectively. In reality, there is no need to specify them as being of one monad over the other, since, being invertible, their inverses are distributive laws in the opposite direction.

\begin{proposition}\label{prop:distrib_laws}
    The isomorphisms of Proposition~\ref{prop:ultra_preserve_coprod} are distributive laws of $\beta$ over $P_E$, and over $A_M$, respectively.
\end{proposition}
\begin{proof}
    We check this for the $P_E$ case, the other being similar. Let $\delta$ denote the isomorphism $\beta P_E \cong P_E \beta$ of Proposition~\ref{prop:ultra_preserve_coprod}. We need to prove the commutativity of four diagrams:
    \[\begin{tikzcd}[column sep=0.2em]
    & \beta \ar[dl, "\beta \eta"'] \ar[dr, "\eta \beta"] \\
    \beta P_E \ar[rr, "\delta"] && P_E \beta \\
    & P_E \ar[ul, "\eta P_E"] \ar[ur, "P_E \eta"']
    \end{tikzcd} \qquad 
    \begin{tikzcd}
    \beta P_E^2 \ar[r, "\delta P_E"] \ar[d, "\beta \mu"] & P_E \beta P_E \ar[r, "P_E \delta"] & P_E^2 \beta \ar[d, "\mu \beta"] \\
    \beta P_E \ar[rr, "\delta"] && P_E \beta \\
    \beta^2 P_E \ar[r, "\beta \delta"] \ar[u, "\mu P_E"'] & \beta P_E \beta \ar[r, "\delta \beta"] & P_E \beta^2 \ar[u, "P_E \mu"']
    \end{tikzcd}\]
    The commutativity of the two triangles and the top pentagon is direct from the definitions; we spell out that of the bottom pentagon. Let $X$ be a set, and $\mathfrak{F} \in \beta^2 P_E X$. Since $E$ is finite, an ultrafilter $\calF$ on $P_E X$ contains $E$ iff $\calF = \pf\{e\}$ for some $e \in E$. Hence,
    \[E^\# \cap \beta P_EX = \{\pf\{e\} \mid e \in E \},\]
    with the notation from \eqref{eq:filter_mult}. There are two cases:
    \begin{enumerate}[(i)]
        \item If $E \in \mu_{P_EX}(\mathfrak{F})$, then $E^\# \cap \beta P_EX \in \mathfrak{F}$. This in turn means that $\mathfrak{F} = \pf \{\pf \{e\}\}$ for some $e \in E$, and that $\delta_X \mu_{P_EX} (\mathfrak{F}) = e$. For the bottom composite, we have
        \[\begin{tikzcd}
            \pf\{\pf\{e\}\} \ar[r, "\beta \delta_X", mapsto] & \pf\{e\} \ar[r, "\delta_{\beta X}", mapsto] & e \ar[r, "P_E \mu_X", mapsto] & e.
        \end{tikzcd}\]
        \item If $E \notin \mu_{P_EX}(\mathfrak{F})$, then $\delta_X \mu_{P_EX} (\mathfrak{F}) = \mu_{P_EX}(\mathfrak{F}) \cap \PS X$. In this case, $E^\# \cap \beta P_EX \notin \mathfrak{F}$. But $\delta^{-1}_X E = E^\# \cap \beta P_EX$, so $E \notin \beta \delta_X(\mathfrak{F})$. In turn, 
        \[\delta_{\beta X}(\beta \delta_X(\mathfrak{F})) = \beta \delta_X(\mathfrak{F}) \cap \PS \beta X \in \beta^2 X.\]
        Applying $\mu_X$ to this, we get an ultrafilter on $X$ that contains $A \subseteq X$ iff
        \[A^\# \cap \beta X \in \beta \delta_X(\mathfrak{F}) \cap \PS \beta X \iff A^\# \cap \beta X \in \beta \delta_X(\mathfrak{F}) \iff \delta^{-1}_X (A^\# \cap \beta X) \in \mathfrak{F}.\]
        On the other hand, 
        \[A \in \mu_{P_EX}(\mathfrak{F}) \cap \PS X \iff A^\# \cap \beta P_E X \in \mathfrak{F}.\] But an ultrafilter $\calF$ on $P_EX$ contains $A$ iff $\delta_X(\calF)$ is an ultrafilter on $X$ containing $A$, so $\delta^{-1}_X(A^\# \cap \beta X) = A^\# \cap \beta P_EX$.
    \end{enumerate}
\end{proof}

These distributive laws give the composite functors $P_E\beta$ and $A_M\beta$ monad structures. We will show that these are the pushforwards of $P_E^\rmf$ and $A_M^\rmf$. 

\begin{theorem}\label{thm:PE_AM_extensions}
    Let $E$ be a finite set and $M$ be a finite monoid. We have isomorphisms of monads
    \[P_E \beta \cong i_\# P_E^\rmf \qquad \text{and} \qquad A_M \beta \cong i_\# A_M^\rmf,\]
    where the monad structures on the left are those induced by the distributive laws of Proposition~\ref{prop:distrib_laws}.
\end{theorem}
\begin{proof}
    Note that, since $\beta$ restricts along $i$ to the identity, $P_E \beta$ and $A_M \beta$ restrict to $P_E^\rmf$ and $A_M^\rmf$ respectively. This gives unit maps $\nu \colon P_E\beta \to i_\# P_E^\rmf$ and $\xi \colon A_M \beta \to i_\# A_M^\rmf$. We show that these are isomorphisms.

    For convenience, we will use the codensity monad description of $\beta$. For the remainder of this proof, $Y$ will always denote a finite set. Let us spell out what $\nu_X$ is doing: given $g \colon X \to Y$, we have $\ang{g} \cdot \nu_X = \ang{g} + E \colon \beta X + E \to Y + E$.

    We now describe an inverse for $\nu_X$. Let $\psi \in i_\#P_E^\rmf(X)$, and let $!_X \colon X \to 1$ be the unique morphism into the terminal object, so $\ang{!_X}(\psi) \in 1 + E$. For any $f \colon X \to Y$, the cone condition applied to the commutative triangle
    \[\begin{tikzcd}
    X \ar[r, "f"] \ar[dr, "!_X"'] & Y \ar[d, "!_Y"] \\
    & 1
    \end{tikzcd}\]
    implies that $\ang{!_X} = P_E !_Y \cdot \ang{f}$. Thus, $\ang{!_X}(\psi) \in E$ iff $\ang{f}(\psi) \in E$, and both are equal if this is the case. Otherwise, $\ang{f}(\psi) \in Y$ for all $f$, and the cone condition shows that $\psi \in \beta X$. This gives a function $i_\# P_E^\rmf(X) \to P_E \beta X$, which is clearly the inverse of $\nu_X$.

    The construction of an inverse to $\xi_X$ is largely similar. One checks that for $\psi \in i_\#A_M^\rmf(X)$, the value of $\ang{!_X}(\psi) \in M$ determines what $M$-indexed component $\ang{f}(\psi)$ lands in for any $f$. Thus, $\psi$ amounts to the data of an ultrafilter on $X$ and an element of $M$.
\end{proof}

The fact that these monads come from distributive laws allow us to easily say what their algebras are. The category of algebras of $i_\# P_E^\rmf$ will be isomorphic to the category of algebras of the lift of $P_E$ to $\CHaus$, and similarly for $i_\# A_M^\rmf$.

\begin{corollary}
    Let $E$ be a finite set and $M$ be a finite monoid. Then:
    \begin{enumerate}[(i)]
        \item $\Set^{i_\# P_E^\rmf}$ is the category of $E$-pointed compact Hausdorff spaces, or, equivalently, $E/\CHaus$ where $E$ is given the discrete topology;
        \item $\Set^{i_\# A_M^\rmf}$ is the category of compact Hausdorff spaces equipped with a discrete left $M$-action and $M$-equivariant continuous maps.
    \end{enumerate}
\end{corollary}

\subsection{Extending the powerset monad}

We now turn our attention to the powerset monad $\PS^\rmf$ on $\FinSet$. This time there is no known distributive law between $\PS$ and $\beta$, although there is a \emph{weak} distributive law of $\PS$ over $\beta$ described by Garner~\cite[p.~349]{Garner2020}. Weak distributive laws were introduced by Street~\cite{Street2009} and B\"ohm~\cite{Bohm2010} as a generalisation of distributive laws where one of the unit axioms is dropped. Under certain conditions on the base category (which $\Set$ satisfies), weak distributive laws correspond to weak liftings of one of the monads to the category of algebras of the other. In our case, Garner's weak distributive law gives a weak lifting of $\PS$ to $\CHaus$, which is the Vietoris monad. It sends a compact Hausdorff space $X$ to its Vietoris hyperspace: the set of closed subspaces of $X$ equipped with a compact Hausdorff topology. Its algebras coincide with the algebras of the filter monad $F$, and they are the continuous lattices, which are special kinds of complete lattices with a compact Hausdorff topology. This makes them the perfect candidates for the algebras of $i_\# \PS^\rmf$, and indeed they will turn out to be. This subsection is devoted to proving that $i_\# \PS^\rmf$ is the filter monad.

\begin{lemma}\label{lem:F_monoid}
    The filter monad $F$ on $\Set$ restricts along $i \colon \FinSet \to \Set$ to $\PS^\rmf$.
\end{lemma}
\begin{proof}
    Recall that if $\calF$ is a filter on a finite set $Y$, then $\bigcap \calF = \bigcap_{A \in \calF} A$ is a finite intersection of elements of $\calF$, and hence is a least element of $\calF$, at which $\calF$ is principal. This shows that the map $\sigma_Y \colon i\PS^\rmf(Y) \to Fi(Y)$ sending $A \subseteq Y$ to $\pf A$ is an isomorphism. It is easily seen to be natural in $Y$, so that $\PS^\rmf$ is a lift of $F$ along $i$. Checking that the monad structures agree amounts showing that $\sigma$ is a colax transformation (see~\eqref{eq:colax}). The condition involving the units is immediate, while the other one follows from a routine calculation.
\end{proof}

This ensures that we have a unit map $F \to i_\# \PS^\rmf = i_\# i^\# F$. The following lemma will allow us to show that it is invertible.

\begin{lemma}\label{lem:TFAE_filter}
    Let $X$ be a set and $\phi \in i_\#\PS^\rmf(X)$. The following are equivalent for $A \subseteq X$:
    \begin{enumerate}[(i)]
        \item \label{part:f_i} $\ang{f}(\phi) \subseteq B$ for all $f \colon X \to Y$ with $Y$ finite, and $B \subseteq Y$ such that $f^{-1}B = A$,
        \item \label{part:f_ii} $\ang{f}(\phi) \subseteq B$ for some $f \colon X \to Y$ with $Y$ finite, and $B \subseteq Y$ such that $f^{-1}B = A$,
        \item \label{part:f_iii} $\ang{\chi_A}(\phi) \subseteq \{\top\}$.
    \end{enumerate}
\end{lemma}
\begin{proof}
    If $f \colon X \to Y$ is such that $Y$ is finite and $f^{-1}B = A$, then the triangle
    \[\begin{tikzcd}
    X \ar[dr, "\chi_A"'] \ar[r, "f"] & Y \ar[d, "\chi_B"] \\
    & 2
    \end{tikzcd}\]
    commutes. The cone condition implies that $\ang{\chi_A} = \PS \chi_B \cdot \ang{f}$, so that $\ang{\chi_A}(\phi) \subseteq \{\top\}$ iff $\ang{f}(\phi) \subseteq B$. Since $f$ was arbitrary, and such an $f$ always exists (e.g.\ $\chi_A$), this shows the equivalence between the three conditions.
\end{proof}

\begin{theorem}\label{thm:powerset_extension}
    The monad $i_\# \PS^\rmf$ is isomorphic to the filter monad.
\end{theorem}
\begin{proof}
    Lemma~\ref{lem:F_monoid} and Theorem~\ref{thm:mnd_adj} give a canonical monad map $\nu \colon F \to i_\# \PS^\rmf$. Explicitly, given a set $X$, a filter $\calF$ on $X$, and $f \colon X \to Y$ with $Y$ finite, we get a filter $Ff(\calF)$ on $Y$, which is principal at $\bigcap Ff(\calF)$. Then, $(\ang{f} \cdot \nu_X)(\calF) = \bigcap Ff(\calF)$. We now construct an inverse to $\nu_X$.
    
    Given $\phi \in i_\# \PS^\rmf(X)$, we get a filter $\calF_\phi$ on $X$ by declaring that $A \in \calF_\phi$ iff it satisfies any of the equivalent conditions of Lemma~\ref{lem:TFAE_filter}. Certainly, $X \in \calF_\phi$ by condition \ref{part:f_iii}, since $\chi_X$ factors through $\{\top\}$. For $A, B \subseteq X$, consider the commutative diagram:
    \[\begin{tikzcd}
    & X \ar[dl, "\chi_A"'] \ar[dr, "\chi_B"] \ar[d, "p"] \\
    2 & 2 \times 2 \ar[r, "\pi_2"'] \ar[l, "\pi_1"] & 2
    \end{tikzcd}\]
    The cone conditions ensures that $\ang{\chi_A} = \PS \pi_1 \cdot \ang{p}$ and $\ang{\chi_B} = \PS \pi_2 \cdot \ang{p}$. Thus, $\ang{p}(\phi) \subseteq \{(\top,\top)\}$ iff $\ang{\chi_A}(\phi)$ and $\ang{\chi_B}(\phi)$ are contained in $\{\top\}$. As $p^{-1}\{(\top,\top)\} = A \cap B$, conditions \ref{part:f_ii} and \ref{part:f_iii} show that $A, B \in \calF_\phi$ iff $A \cap B \in \calF_\phi$.

    Lastly, we check that this provides an inverse to $\nu_X$. Given $\phi \in i_\# \PS^\rmf (X)$, for $f \colon X \to Y$ with $Y$ finite, $(\ang{f} \cdot \nu_X)(\calF_\phi)$ is the smallest $B \subseteq Y$ such that $f^{-1}B \in \calF_\phi$. But we have
    \[f^{-1}B \in \calF_\phi \iff \ang{f}(\phi) \subseteq B,\]
    the forward implication coming from Lemma~\ref{lem:TFAE_filter} part \ref{part:f_i}, and the backwards one from part \ref{part:f_ii}. Therefore, $(\ang{f} \cdot \nu_X)(\calF_\phi) = \ang{f}(\phi)$, as needed. In the other direction, for $\calF \in FX$ and $A \subseteq X$ we have
    \begin{align*}
        A \in \calF_{\nu_X(\calF)} &\iff (\ang{\chi_A} \cdot \nu_X)(\calF) \subseteq \{\top\} \\
        &\iff \{\top\} \in F\chi_A(\calF) \\
        &\iff \chi_A^{-1}\{\top\} = A \in \calF.
    \end{align*}
\end{proof}

The category of $F$-algebras was identified by Day~\cite[Thms.~3.3 and~4.5]{Day1975} as the category of continuous lattices and maps preserving directed joins and arbitrary meets. 

\begin{corollary}
    $\Set^{i_\# \PS^\rmf}$ is the category of continuous lattices.
\end{corollary}

The facts that a continuous lattice is a special kind of complete lattice, and that it has a canonical compact Hausdorff topology are then automatic from the pushforward construction, by taking $T = \PS$ in~\eqref{eq:monadic_square}.

\subsection{The rank of pushforwards}\label{ssc:rank}

In this last subsection we will show that, with the exception of Examples~\ref{ex:easy_pf_finset}\ref{ex:pff:T} and~\ref{ex:pff:S}, namely the cases of the two inconsistent monads, pushforwards along $i$ never have rank. 

This will be an easy consequence of the result that, for any regular cardinal $\lambda$, any subfunctor of a $\lambda$-accessible endofunctor of $\Set$ is $\lambda$-accessible. This can be proved from the fact that $\Set$ is strictly locally finitely presentable in the sense of Ad\'{a}mek, Milius, Sousa and Wi{\ss}mann~\cite[Def.~3.9]{Adamek2019}, and that, between such categories, $\lambda$-accessible functors coincide with $\lambda$-bounded ones~\cite[Thm.~4.11]{Adamek2019}. However, we give an elementary direct proof here.

\newcommand{\Sub}{\mathsf{Sub}}

\begin{proposition}\label{prop:acc}
    Let $\lambda$ be a regular cardinal, and $T \colon \Set \to \Set$ be $\lambda$-accessible functor. Any subfunctor of $T$ is $\lambda$-accessible.
\end{proposition}
\begin{proof}
    We will use the fact that a functor out of $\Set$ is $\lambda$-accessible iff, for all sets $X$, it preserves the colimit of the canonical diagram $D_X \colon \Sub_\lambda(X) \to \Set$, where $\Sub_\lambda(X)$ is the poset of subobjects of $X$ of cardinality less than $\lambda$. A proof of this (in the finitary case) can be found in~\cite[Cor.~2.7]{Adamek2019}. Recall that a colimit cocone for $D_X$ is $(p \colon Y \hookrightarrow X)$ indexed by $p \in \Sub_\lambda(X)$.

    Let $m \colon S \to T$ be a monic natural transformation, and consider the commutative square
    \[\begin{tikzcd}
    \colim SD_X \ar[r] \ar[d, hook, "\colim mD_x"'] & SX \ar[d, hook, "m_X"] \\
    \colim TD_X \ar[r, "\cong"] & TX.
    \end{tikzcd}\]
    The injectivity of the function on the left follows the construction of $\lambda$-filtered colimits in $\Set$, and the fact that $mD_X$ is pointwise injective. Hence, the map on the top is injective, so it suffices to show that it is surjective.

    This is trivial if $X = \varnothing$, so we assume otherwise. Let $s \in SX$. Then $m_X(s) = Tp(t)$ for some $p \colon Y \hookrightarrow X$ in $\Sub_\lambda(X)$ and $t \in TY$. Without loss of generality, we may assume that $Y$ is nonempty, so that $p$ has a retraction, say $rp = 1_Y$. We get two commuting squares, the inside and the outside faces of this diagram:
    \[\begin{tikzcd}
    SY \ar[r, "Sp"'] \ar[d, hook, "m_Y"'] & SX \ar[d, hook, "m_X"] \ar[l, bend right, "Sr"'] \\
    TY \ar[r, "Tp"] & TX \ar[l, bend left, "Tr"]
    \end{tikzcd}\]
    Then
    \begin{align*}
        m_X \cdot Sp \cdot Sr(s) &= Tp \cdot m_Y \cdot Sr(s) \\
        &= Tp \cdot Tr \cdot m_X(s) \\
        &= Tp \cdot Tr \cdot Tp(t) \\
        &= Tp(t) \\
        &= m_X(s),
    \end{align*}
    and hence $Sp \cdot Sr(s) = s$. This shows that $s$ is in the image of $\colim SD_X \to SX$, and, since $s$ was arbitrary, that this map is surjective.
\end{proof}

\begin{remark}
    This proposition is crucially limited to functors whose domain is $\Set$. For example, the inclusion $\Z \to \Q$ is an epimorphism in $\Ring$, and since $y \colon \Ring^\opp \to [\Ring,\Set]$ preserves limits, it gives a monomorphism $\Ring(\Q,-) \to \Ring(\Z,-)$. However, as rings, $\Z$ is finitely presentable, while $\Q$ is not even finitely generated, so $\Ring(\Z,-)$ is $\aleph_0$-accessible, but $\Ring(\Q,-)$ is not.
\end{remark}

\begin{theorem}
Let $T$ be a consistent monad on $\FinSet$ (i.e.\ $\eta^T$ is monic). Then $i_\# T$ does not have rank.
\end{theorem}
\begin{proof}
Since $\FinSet$ and $\Set$ have kernel pairs, monomorphisms in $\Mnd_\FinSet$ and $\Mnd_\Set$ are exactly those whose underlying natural transformation is monic (see Remark~\ref{rmk:monic}), which coincide with those which are pointwise monic. As $i$ preserves monomorphisms, both $i_* \colon [\FinSet,\FinSet] \to [\FinSet, \Set]$ and $\Ran_i$ (being a right adjoint) preserve monomorphisms, and then so does $i_\#$. It follows that $i_\# \eta^T \colon \beta \to i_\#T$ is a monic natural transformation. Since the ultrafilter monad does not have rank, Proposition~\ref{prop:acc} then implies that neither does $i_\#T$.

That the ultrafilter monad does not have rank is stated in several sources, e.g.\ in Borceux's book~\cite[Vol.~II, p.~235]{Borceux1994}, but seemingly always without proof. Here is a sketch of one: for a regular cardinal $\lambda$ and a set $X$ of cardinality at least $\lambda$, the collection of subsets of $X$ whose complement has cardinality less than $\lambda$ is a proper filter. By the ultrafilter lemma, it is contained in some ultrafilter $\calF$. One checks that $\calF$ is not in the image of $\beta m$ for any $m \in \Sub_\lambda(X)$. It follows that the map $\colim \beta {D_X} \to \beta X$ in the proof of Proposition~\ref{prop:acc} is not surjective, and hence that $\beta$ is not $\lambda$-accessible.
\end{proof}

\bibliography{C:/Users/drnai/OneDrive/Documents/References/monads}

\end{document}